\documentclass[12 pt]{amsart}
\usepackage{amssymb,pdfsync}
\usepackage{graphicx}
\usepackage[mathscr]{eucal}
\usepackage{dsfont}
\usepackage[usenames]{color}
\definecolor{Purple}{rgb}{.7,0.08,0.6} 
 
\usepackage{enumerate}

\theoremstyle{plain}
\newtheorem{Thm}{Theorem}
\newtheorem{Cor}[Thm]{Corollary}
\newtheorem{Prop}[Thm]{Proposition}
\newtheorem{Lem}[Thm]{Lemma}

\theoremstyle{definition}

\newcommand{\B}{\mathbf B}

\newcommand{\Bt}{\tilde\B}
\renewcommand{\d}{{\tt{d}}}

\newcommand{\Dep}{\mathscr{L}_\epsilon}
\newcommand{\Do}{\mathcal{L}_0}
\renewcommand{\l}{\lambda}
\newcommand{\g}{g}

\newcommand{\G}{\eta}

\renewcommand{\ge}{g_\epsilon}
\newcommand{\go}{g_0}
\renewcommand{\r}{\rho}
\newcommand{\p}{\gamma}

\newcommand{\Atr}{\mathcal A}
\newcommand{\Atre}{\Atr_\epsilon}

\renewcommand{\Re}{\operatorname{Re}}
\renewcommand{\Im}{\operatorname{Im}}
\renewcommand{\bar}{\overline}
\renewcommand{\tilde}{\widetilde}

\newcommand{\Atrei}{\mathcal B^{\iota}_\epsilon}
\newcommand{\C}{\mathbb C}

\newcommand{\Cin}{\mathbf C}
\newcommand{\CinI}{\mathbf C^1}

\newcommand{\CinII}{\mathbf C^2}

\newcommand{\Cins}{\mathbf C^\sharp}

\newcommand{\Cine}{\mathbf{C}_\epsilon}
\newcommand{\Cker}{C}
\newcommand{\CkerI}{C^1}
\newcommand{\CkerII}{C^2}

\newcommand{\Ctr}{\mathcal C}

\newcommand{\Ctre}{\mathcal C_\epsilon}
\newcommand{\Ctred}{\Ctr_{\epsilon}^{\dagger}}
\newcommand{\Ctrei}{\Ctr_{\epsilon}^{\iota}}
\newcommand{\Ctreid}{(\Ctr_{\epsilon}^{\iota})^\dagger}

\newcommand{\Ctrs}{\mathcal C^\sharp}
\newcommand{\Ctrse}{\mathcal C^\sharp_\epsilon}
\newcommand{\Ctrsd}{\mathcal C^{\sharp, \delta}}
\newcommand{\Ctrsei}{\mathcal C^{\sharp, s}_\epsilon}
\newcommand{\Dsei}{\mathcal A^{s}_\epsilon}

\newcommand{\Ein}{\mathbf{E}}

\newcommand{\Etr}{\mathcal E}
\newcommand{\Etre}{\mathcal E_\epsilon}
\newcommand{\Etrei}{\mathcal R_{\epsilon}^{\iota}}
\newcommand{\Etreid}{(\mathcal R_{\epsilon}^{\iota})^\dagger}
\newcommand{\h}{\mathfrak h}
\newcommand{\ha}{\mathfrak h^*}
\newcommand{\hda}{\mathfrak h_\delta^*}

\newcommand{\hp}{\mathcal{H}^p}

\newcommand{\htwom}{\mathcal{H}^2(\bd D, \omega\,d\sigma)}

\renewcommand{\iota}{s}
\newcommand{\Ksei}{A^{s}_\epsilon}

\newcommand{\Rin}{\mathbf{R}}

\newcommand{\Rins}{\mathbf{R}^\sharp}

\newcommand{\Rtr}{\mathcal R}

\newcommand{\Rtrei}{\mathcal R_\epsilon^\iota}

\newcommand{\Rtrs}{\mathcal R^\sharp}
\newcommand{\Rtrse}{\mathcal R^\sharp_\epsilon}
\newcommand{\Rtrsei}{\mathcal R^{\sharp, \iota}_\epsilon}
\renewcommand{\S}{\mathcal S}

\newcommand{\Sone}{\mathcal S_1}
\newcommand{\Sl}{\mathcal S_{\!\,\lambda}}
\newcommand{\Sw}{\mathcal S_{\!\,\omega}}
\newcommand{\Swd}{\mathcal S_{\!\,\omega}^\dagger}

\newcommand{\Td}{T^\dagger}

\newcommand{\R}{\mathbb R}

\newcommand{\Cn}{\mathbb C^n}

\newcommand{\dee}{\partial}

\newcommand{\deebar}{\overline\partial}

\newcommand{\eps}{\epsilon}
\newcommand{\bndry}{b}

\newcommand{\wh}{\hat{w}}

\newcommand{\zp}{z}
\numberwithin{equation}{section}

\newcommand{\bd}{\bndry}

\begin{document}
\title
[Cauchy-Szeg\H o Projection]{The Cauchy-Szeg\H o projection
 for domains in $\Cn$ with minimal smoothness}
\author[Lanzani and Stein]{Loredana Lanzani$^*$
and Elias M. Stein$^{**}$}
\thanks{$^*$ Supported in part by the National Science Foundation, awards DMS-1001304 and 
DMS-1503612.}
\thanks{$^{**}$ Supported in part by the National Science Foundation, awards
DMS-0901040 and DMS-1265524.}
\address{
Dept. of Mathematics,       
Syracuse University 
Syracuse, NY 13244-1150 USA}
\email{llanzani@syr.edu}
\address{
Dept. of Mathematics\\Princeton University 
\\Princeton, NJ   08544-100 USA }
  \email{stein@math.princeton.edu}
\begin{abstract} 
We prove $L^p(\bndry D)$-regularity of the Cauchy-Szeg\H o pro-

\noindent jection$^1$ for bounded domains $D\subset\Cn$ whose boundary satisfies the minimal regularity condition of class $C^2$, together with a naturally occurring notion of convexity. 
\end{abstract}
\maketitle

\section{Introduction}
The purpose of this paper is to prove the $L^p$-regularity of the Cauchy-Szeg\H o projection\footnote{also known as {\em Szeg\H o projection}.}
for domains $D \subset \Cn$
  with $n\geq 2$,
whose boundary is subject to minimal smoothness hypotheses. 
In recent years, the
 study of the 
basic domain operators, 
such as the Cauchy integral and the Bergman projection, has been 
undertaken
in the context of minimal smoothness.
 However, the corresponding question for the
Cauchy-Szeg\H o projection, which raises a number of serious different issues, 
has hitherto not been broached.
\subsection{Background} 
The Cauchy-Szeg\H o projection $\S$ is defined as the orthogonal projection of $L^2(\bndry D)$
onto the holomorphic Hardy space $\hp (\bndry D)$, 
which is the closure of the subspace of functions on $\bndry D$ that arise as restrictions of functions continuous on $\bar D$ and holomorphic in $D$. A list of some earlier relevant papers includes \cite{PS}, \cite{KS-2}, \cite{AS-1}, \cite{BoLo}, \cite{NRSW}, \cite{Cu}, \cite{H}, \cite{FH-1}, \cite{FH-2}, \cite{MS-2}, \cite{CD}, \cite{Ko-1}, \cite{HNW}.

\smallskip

$\bullet$ Let us first recall some relevant facts for $\Cn$ when $n=1$ (the planar setting). In the special
case when $D$ is the unit disc, $\S$ is in fact the Cauchy integral, and its $L^p$ estimate
is the classical theorem of M. Riesz. For more general $D\subset\C$
the matter can be briefly put as follows: the lower limit of smoothness of $\bndry D$ that can assure the $L^p$-boundedness of $\S$ is of one order of differentiability. More precisely, if
$\bndry D$ is of class $C^1$, then $\S$ is bounded on 
$L^p(\bndry D)$,
for $1<p<\infty$. 
However if $\bndry D$ is merely Lipschitz, then the conclusion holds for a range $p_D<p<p_D'$, where $p_D$ depends on the Lipschitz bound of $D$, and in any case $1<p_D<4/3$. There are two ways of achieving these 
results. The first is by conformal mapping (hence essentially for simply connected domains)
see \cite{S}, \cite{B}, \cite{LS-1} and references therein. The second depends on the 
$L^p$-theory of the Cauchy integral $\Cin$ and its boundary transform $\Ctr$, see \cite{C} and \cite{CMM}, and proceeds via the identity
\begin{equation}\label{E:I1}
\Ctr = \S(I-\Atr), \quad \mathrm{with}\ \ \Atr = \Ctr^*-\Ctr\, ,
\end{equation}
where $\Ctr^*$ is the adjoint of $\Ctr$ on
$L^2(\bndry D)$,
and $I$ is the identity operator. The issue then becomes the possible invertibility of $I-\Atr$ in 
$L^2(\bndry D)$, see \cite{KS-1}, \cite{LS-1}.
\smallskip

$\bullet$ Turning to the case when $D\subset \Cn$ with $n>1$, one is immediately faced with 
several obstacles not seen when $n=1$. 

{\em(a)}\quad The fact that the requirement of pseudo-convexity of the domain $D$ must
necessarily arise. Since pseudo-convexity is a notion essentially bearing on 
the second fundamental form of $\bndry D$, it is reasonable to expect minimal smoothness
to be ``near $C^2$'', as opposed to ``near $C^1$'' for $n=1$.

{\em(b)}\quad The analogue of the approach via conformal mapping is not viable for multiple
reasons, one of which is that holomorphic equivalence of domains in $\Cn$ is highly restrictive
when $n>1$, and therefore not applicable to general classes of domains.

{\em(c)}\quad The possible use of an identity like \eqref{E:I1} is problematic, because when $n>1$ there are infinitely many Cauchy integrals that present themselves, while no one seems appropriate for a direct use of \eqref{E:I1} (unless $D$ is in fact relatively smooth \cite{KS-2}).

{\em (d)}\quad It would be possible to prove the $L^p$-boundedness of $\S$ by the Calder\`on-Zygmund paradigm if we had a satisfactory description of the kernel of this operator. However, the asymptotic formula of Fefferman which would do this (analogous to his well-known description of the Bergman kernel \cite{F}) requires that the domain be relatively smooth, which is not the case in what follows below.
\subsection{Main result}\label{SS:MR}
 To state our result we need to make the definition of the Cauchy-Szeg\H o projection precise.  If $d\sigma$ denotes the induced Lebesgue measure on $\bndry D$, 
we write $L^p(\bndry D)$ for $L^p(\bndry D, d\sigma)$. In addition, if
$d\omega$ is any measure on $\bndry D$ of the form $d\omega =\omega d\sigma$, where the density $\omega$ is a strictly positive continuous function on $\bndry D$, then one can consider 
$L^p(\bndry D, d\omega)$, but note that this space contains the same elements as
$L^p(\bndry D, d\sigma)$ and the two norms are equivalent. Thus we will continue to denote both of these spaces by $L^p(\bndry D)$. However, the distinction between $L^2(\bndry D, d\sigma)$ and $L^2(\bndry D, d\omega)$ become relevant when defining the Cauchy-Szeg\H o projection, because these spaces have different inner products that give different notions of orthogonality. 
So
the Cauchy-Szeg\H o projection $\Sw$ is the orthogonal 
projection of $L^2(\bndry D, \omega d\sigma)$ onto $\htwom$. Note that $\Sone$ is the ``$\S$'' discussed above, but there is no simple connection between the general $\Sw$ and $\Sone$. Nevertheless, the case when $d\omega$ is the Leray-Levi measure that arises below, is key to understanding the general result. It states:
\smallskip

{\em Suppose $D\subset\Cn$, $n\geq 2$, is a bounded domain whose boundary is of class $C^2$ and is strongly pseudo-convex. Then $\Sw$ is a bounded operator on $L^p(\bndry D, d\omega)$, for $1<p<\infty$.}
\smallskip

There are five main steps in the proof. 
\smallskip

{\em Step 1: Cauchy integrals.}\quad For our purposes a Cauchy integral $\Cin$ is an operator mapping functions on $\bndry D$ to functions on $D$ given by
\begin{equation}\label{E:I2}
\Cin (f)(z) =\int\limits_{w\in\bndry D}\!\!\! f(w)\, C(w, z),\quad z\in D,
\end{equation}
with the following properties:
\begin{itemize}
\item[(i)] $\Cin$ produces holomorphic functions. More precisely, if $f$ is integrable on $\bndry D$ then $\Cin (f)$ is holomorphic in $D$.
\item[]
\item[(ii)] $\Cin$ reproduces holomorphic functions. That is,
$$
\Cin (f)(z) = F(z),\quad z\in D,
$$
whenever $F$ is holomorphic in $D$ and continuous on $\bar D$, and 
$\displaystyle{F\big |_{\bndry D} =f}$.
\item[]
\item[(iii)]
  The kernel $C(w, z)$
  is 
  ``explicit'' enough to allow ``relevant'' computations.
\end{itemize}
A general method for obtaining integrals that satisfy the requirements (ii) and (iii) (but not necessarily (i)) is that of the Cauchy-Fantappi\` e formalism, see e.g.,  \cite{LS-3}. The point of departure
is a ``generating'' form
$$
G(w, z) = \sum\limits_{j=1}^nG_j(w, z)\, dw_j
$$
that satisfies
\begin{equation}\label{E:I2a}
g(w, z)=\langle G(w, z), w-z\rangle \neq 0
\end{equation}
if $z\in D$ and $w$ is in a small neighborhood of $\bndry D$. (Here, 
$\langle G(w, z), w-z\rangle$ denotes the action of $G(w, z)$ on the vector $w-z$.)
 Then the corresponding Cauchy integral is defined by \eqref{E:I2}, with
 $$
 C(w, z) = \frac{1}{(2\pi i)^n}\frac{G\wedge (\deebar_w G)^{n-1}}{g(w, z)^n}\, .
 $$
 One observes that when $n=1$, there is only one such integral kernel (namely the familiar Cauchy kernel
 $C(w, z) = dw/2\pi i(w-z)$), while for $n\geq 2$ there are infinitely many such kernels. Moreover, the existence of a generating form that satisfies property (i.) is closely related (by way of the so-called Levi problem)
 to the requirement that $D$ be pseudo-convex, and such forms have been constructed only when $D$ is actually strongly pseudo-convex (and relatively smooth), see \cite{H}, \cite{R}, \cite{KS-2}. These constructions take for $z\in D$ near $w\in\bndry D$ and $\rho$  a defining function of $D$
 \begin{equation}\label{E:I3}
 G(w, z) =\dee\rho (w) -\frac12\sum\limits_{j, k=1}^n\frac{\dee^2\rho (w)}{\dee w_j\dee w_k}(w_k-z_k)\, dw_j\, ,
 \end{equation}
 and then extend $G(w, z)$ to $z\in D$ by differing methods. It is to be noted that the function 
 $g(w, z)$ that results from this choice of $G(w, z)$ via the construction \eqref{E:I2a} is the Levi polynomial at $w\in\bndry D$.
 
 Now if we take \eqref{E:I3} as our starting point we immediately run into a first obstacle. Since we have assumed that $D$ (that is the function $\rho$) is of class $C^2$, the denominator $\g^n$ in $C(w, z)$ above, cannot be guaranteed any degree of smoothness in $w$, beyond continuity in that variable. But all known methods of proving $L^2$ (or $L^p$) boundedness 
 of singular integrals (that is, the $T(1)$ theorem and its variants) require some degree of
 smoothness of the kernel away from the diagonal $\{w=z\}$. One way to get around this difficulty is to replace the matrix $\displaystyle{\{\dee^2\rho/\dee w_j\dee w_k\}}$ appearing in
 \eqref{E:I3} by a $C^1$-smooth matrix (as is done in \cite{Ra}). However the Cauchy integral constructed this way would not be of any substantial use to us. What we require below 
 is a {\em family} of Cauchy integrals $\{\Cine\}_\epsilon$, where for each $\epsilon$ we have replaced $\displaystyle{\{\dee^2\rho/\dee w_j\dee w_k\}}$ by an appropriate $C^1$-smooth matrix $\displaystyle{\{\tau^\epsilon_{jk}\}}$ with uniform error on $\bndry D$ (less than $\epsilon$), and $\g$ is replaced by the corresponding $\{\ge\}_\epsilon$. However this approximation comes at a price, which increases as $\epsilon\to 0$, as we will see below. 
 
 {\em Step 2:  $L^p(\bndry D)$-regularity of the $\Ctre$.}\quad We apply the machinery of the general $T(1)$ theorem to to the operators $\Ctre$, which are the boundary restrictions of the Cauchy integrals $\Cine$, and to do this requires four things. First, we need
 a space of 
 homogeneous type that reflects the non-isotropic geometry of $\bndry D$.
 Second, we need to establish the difference inequalities
 for the kernel of $\Ctre$, and these can be achieved because we have replaced the matrix
  $\displaystyle{\{\dee^2\rho/\dee w_j\dee w_k\}}$ by a suitable approximation $\{\tau^\epsilon_{jk}\}_\epsilon$. The third item is the analysis of the  formal adjoint of $\Ctre$ with respect to the inner product of $L^2(\bndry D, d\l)$. Here $d\l$ is the Leray-Levi measure, defined as integration on $\bndry D$ with respect to the
 ($2n-1$)-form $\displaystyle{\dee\rho\wedge(\deebar\dee\rho)^{n-1}/(2\pi i)^n}$.
The analysis of such adjoint is done by identifying the ``essential part'' of $\Ctre$, which is given by $\Ctrse$, with
 $$
 \Ctrse (f) (z) =\int\limits_{w\in\bndry D}\!\!\frac{f(w)}{\ge (w, z)^n}\, d\l (w)\, .
 $$
  
 The fourth item is the proof of the required ``cancellation'' properties of of $\Ctrse$, and these are expressed
 as the action of  $\Ctrse$ on ``bump functions''. For this the key observation is that, unlike the case $n=1$, whenever $f\in C^1(\bndry D)$, then
 $$
 \Ctrse (f) = \Etre (df) + \Rtrse (f)
 $$
 where both the kernels of $\Etre$ and $\Rtrse$ have a singularity weaker by one order than 
 that of $\Ctrse$. 
 This concludes the proof of the regularity of $\Ctre$ on $L^p(\bndry D, d\l)$ for any 
 $1<p<\infty$. (We will return to the general case: $L^p(\bndry D, \omega d\sigma)$ in Step 5 below.)
 \smallskip
 
 {\em Step 3: Relating the Cauchy-Szeg\H o projection to the $\Ctre$'s.}\quad
 At this point we establish an analogous formulation of the original identity \eqref{E:I1} for the Cauchy-Szeg\H o projection relative to the space $L^2(\bndry D, d\l)$, which we denote $\S_{\l}$, namely the identity:
 \begin{equation}\label{E:Ia}
 \S_{\l} (I-\Atre) = \Ctre
 \end{equation}
 where $\Atre = \Ctre^*-\Ctre$, and $^*$ denotes the adjoint in $L^2(\bndry D, d\l)$. 
 At this stage we rely on some results on the Hardy space $\mathcal{H}^2 (\bndry D)$ which will appear separately in \cite{LS-5}. 

 \smallskip
 
 {\em Step 4: Proving the boundedness of $\S_\lambda$ on $L^p(\bndry D, d\l)$.}\quad 
 One would then like to invert the operator $I-\Atre$ that appears in \eqref{E:Ia} by a partial Neumann series,
but the fact is that for our $C^2$-smooth domains, the quantity
 $\|\Atre\|$ is unbounded as $\epsilon \to 0$, with $\|\cdot\|$ the operator norm acting on 
 any $L^p(\bndry D)$ space. This is the price we have to pay for replacing the original matrix
 $\displaystyle{\{\dee^2\rho/\dee w_j\dee w_k\}}$ with the smoother matrices $\{\tau^\epsilon_{jk}\}_\epsilon$. To surmount this difficulty we truncate our operator and write for any
 $s>0$ 
 \begin{equation*}
 \Ctre = \Ctrei +\Rtrei\, ,
 \end{equation*}
 where $\Ctrei$ has the same kernel as $\Ctre$, except that it is appropriately cut off to be supported in $\d (w, z)\leq s$, where $\d (w, z)$ is the non-isotropic quasi-distance for
  the space of homogeneous type described in Step 2 above.
 What this truncation achieves is the following important fact: if 
 $s$ is sufficiently small in terms of
 $\epsilon$, then 
 \begin{equation}\label{E:I5}
 \|(\Ctrei)^*-\Ctrei\|_{L^p\to L^p}
  \lesssim \epsilon^{1/2}M_p.
 \end{equation}
 (Here the adjoint is again taken with respect to $L^2(\bndry D, d\l)$, with $d\l$ the Leray-Levi measure.) This feature of $\Ctrei$ allows us to express \eqref{E:Ia} in the equivalent form
 \begin{equation} \label{E:I1b}
  \S_{\l} = \big(\Ctre +\S_{\l}\big((\Etrei)^* - \Etrei\big)\big)\, \big(I-\big((\Ctrei)^*-\Ctrei\big)\big)^{-1}, 
 \end{equation}
which is valid for $\epsilon$ sufficiently small in terms of $p$, and for
 $s$ sufficiently small in terms of $\epsilon$.  
 The complementary fact is that while the quantity
 $\|(\Rtrei)^*-\Rtrei\|_{L^p\to L^p}$ is not small (in fact, in general this is unbounded as $\epsilon\to 0$), one has the redeeming property that each of $\Rtrei$ and $(\Rtrei)^*$ maps: 
 $L^1(\bndry D, d\l)$ to $L^\infty (\bndry D)$ (in fact to $C(\bndry D)$) for each $\epsilon$ and $s$. 
  Taking all this into account
  we conclude from \eqref{E:I1b} that $\S_{\l}$ is bounded on $L^p(\bndry D, d\l)$ for each 
  $1<p<\infty$.
 \smallskip
 
 {\em Step 5: Passage to $\Sw$ for general $\omega$.} \quad First we note that the boundedness of the $\Ctre$ on $L^p(\bndry D, d\l)$ immediately gives their boundedness on
 $L^p(\bndry D, \omega d\sigma)$ via the remarks that we made before.  While the corresponding result 
 for the Cauchy-Szeg\H o projections  cannot be obtained in the same way, 
 the main idea for $\Sw$  is as follows. 
 If $^\dagger$ denotes the adjoint with respect to the inner product of $L^2(\bndry D, \omega d\sigma)$, and if we write $d\omega =\varphi\, d\l$, then the fact that $(\Ctrei)^\dagger =\varphi (\Ctrei)^*\varphi^{-1}$, and the continuity of $\varphi$, allow us to obtain from \eqref{E:I5} that
 \begin{equation}\label{E:I6}
 \|(\Ctrei)^\dagger-\Ctrei\|_{L^p\to L^p}
  \lesssim \epsilon^{1/2}M_p.
 \end{equation}
 With these things in place one can then proceed as in Step 5.
 
 It is worthwhile  to point out that the original Cauchy-Szeg\H o projection $\S_1$ (defined with respect to
 the induced Lebesgue measure $d\sigma$) is included here, because it turns out that 
 $d\sigma=\varphi_0\,d\l$ where the density $\varphi_0$ is continuous and positive on account
 of the strong pseudo-convexity and $C^2$-regularity of $D$. 
 \subsection{Organization of the paper} 
 This paper consists of two parts. Part I deals with the Cauchy integrals, and it includes sections 2 through 4.
 Sections \ref{S:2} and \ref{S:3} contain a review of preliminary facts that are needed.
 Since in the main these were also used in our work on the Bergman projection \cite{LS-2} and the Cauchy-Leray integral \cite{LS-3} and \cite{LS-4}, the details of the proofs are for the most part omitted. Section \ref{S:4} contains the $L^p$-theory of the Cauchy integrals $\Cine$ and the corresponding boundary transforms $\Ctre$.
 
  In Part II we turn to the Cauchy-Szeg\H o projection.
  In Section \ref{S:5} we prove
 the key fact \eqref{E:I5} (Proposition \ref{P:6.2.2}), which proves the main result in the context of the Leray-Levi measure (Theorem \ref{T:6.1.1}). Section \ref{S:last} then concludes by dealing with the general case (Theorem \ref{T:6.1.2}). The argument deducing \eqref{E:I6} from \eqref{E:I5} rests on a general result (Lemma \ref{L:7.3}) involving operators whose kernels have support appropriately close to the diagonal.
 \smallskip
 
 {\em Acknowledgement.}\ We are grateful to Po-Lam Yung for 
 proposing that we consider a more general form of Theorem \ref{T:6.1.1} that ultimately led us
 to Theorem \ref{T:6.1.2}.
\section*{Part I: Cauchy Integrals}\label{PI}

In this first part we undertake the study of a family $\{\Cine\}$ of Cauchy integrals that depend on a parameter
 $\epsilon$, and which are constructed using the Cauchy-Fantappi\`e formalism. Here we focus on the properties of $\Cine$ for fixed $\epsilon$. What happens when $\epsilon$ varies, in particular the behavior of $\Cine$ as $\epsilon \to 0$, will be studied in Part II, where this will play a key role in the understanding of the Cauchy-Szeg\H o projection.

\section{The fundamental denominators $\go$ and $\ge$}\label{S:2}
A number of results needed below that are known (see \cite{LS-3}, \cite{LS-4}, \cite{Ra}) are restated here without proof. 
 An exception is Proposition \ref{P:1} and its corollary. 

\subsection{The functions $\go$ and $\ge$}\quad
We consider a bounded domain $D$ in $\Cn$ with defining function $\r:\Cn \to \R$ of class $C^2$, for which
$D=\{\r<0\}$, and $|\nabla\r|>0$ where $\r=0$. We assume that $\r$ is strictly plurisubharmonic. 
The assumptions regarding the domain $D$ and $\r$ will be in force throughout, and so will not be restated below.

We let $\Do (w, z)$ be the negative of the Levi polynomial at $w\in\bndry D$, given by
\begin{equation*}
\Do (w, z) = \langle \dee\r (w), w-z\rangle -\frac12\sum\limits_{j, k}
\frac{\dee^2\r(w)}{\dee w_j\dee w_k}(w_j-z_j)(w_k -z_k)
\end{equation*}
where $\dee\r (w)$ is the 1-form $\sum\r_{w_j}\!(w) dw_j$,
and the
expression $\langle \dee\r (w), w-z\rangle$ denotes the action of $\dee\rho (w)$ on the vector $w-z$, that is
\begin{equation*}
\langle \dee\r (w), w-z\rangle = \sum\limits_j\frac{\dee\r (w)}{\dee w_j} (w_j-z_j)\, .
\end{equation*}
The strict plurisubharmonicity of $\r$ implies that 
\begin{equation*}
2\Re\Do (w, z)\geq -\r (z) + c|w-z|^2,\quad \mbox{for some}\ c>0,
\end{equation*}
whenever $w\in\bndry D$ and $z\in\bar D$ is sufficiently close to $w$. To ensure that this inequality may
 hold globally we make our first modification of $\Do$ and replace it with $\go$ given as 
 \begin{equation*}
 \go (w, z) =\chi \Do + (1-\chi)|w-z|^2.
 \end{equation*}
Here $\chi =\chi (w, z)$ is a $C^\infty$ cut-off function with $\chi =1$ when $|w-z|\leq \mu/2$ and 
$\chi =0$ if $|w-z|\geq\mu$. Then if $\mu$ is chosen sufficiently small (and then kept fixed throughout)
 we have that
 \begin{equation}\label{E:2.2}
 \Re \go(w, z) \geq c\,(-\r (z) + |w-z|^2)\quad 
 \end{equation}
for $z$ in $\bar D$ and $w$ in $\bndry D$, with $c$ a positive constant. 
 
 The modified Levi polynomial $\go$ is not yet quite right for our needs, because
 in general it has no smoothness beyond continuity in the variable $w$. So for each $\epsilon >0$ we consider a variant $\ge$ defined as follows. We find an $n\times n$ matrix $\{\tau^\epsilon_{jk}(w)\}$ of $C^1$ functions
 so that
 \begin{equation*}
 \sup\limits_{w\in\bndry D}
 \left|\frac{\dee^2\r(w)}{\dee w_j\dee w_k} - \tau^\epsilon_{jk}(w)\right| \leq \epsilon, \quad \mbox{for}\quad 1\leq j, k\leq n\, .
 \end{equation*}
 We then set 
 \begin{equation*}
 \Dep (w, z) =\langle\dee\r (w), w-z\rangle -\frac12\sum\limits_{j, k} \tau^\epsilon_{jk}(w)\, (w_j-z_j) (w_k-z_k)\, ,
 \end{equation*}
 and define
 \begin{equation}\label{E:2.3}
 \ge (w, z) = \chi\Dep + (1-\chi)|w-z|^2\quad  \mbox{for}\ z\, , w\in \Cn\, .
 \end{equation}
 Now $\ge$ is of class $C^1$ in $w$ (it is of class $C^\infty$ in $z$). 
 We note that 
 \begin{equation*}
 |\go (w, z) - \ge (w, z)|\lesssim \epsilon |w-z|^2\quad \mbox{for}\ w\in\bndry D\, ,
 \end{equation*}
 and hence if $\epsilon$ is taken sufficiently small (in terms of the constant $c$ appearing in \eqref{E:2.2})
  then automatically
  \begin{equation}\label{E:2.4}
  \Re \ge (w, z) \geq c'(-\r (z) + |w-z|^2),\quad \mbox{for}\ z\in\bar D,\ w\in\bndry D\, ,
  \end{equation}
  for an appropriate positive $c'$. 
  There is also the variant
  \begin{equation*}
  \Re \ge (w, z) \geq c'(\r (w)-\r (z) + |w-z|^2)
   \end{equation*}
  for $z$ and $w$ in a neighborhood of $\bndry D$. 
  We shall always assume that $\epsilon$ is restricted to be so small that 
  \eqref{E:2.4} holds. As a direct consequence of \eqref{E:2.2} we then also have
  \begin{equation}\label{E:2.5}
  |\ge (w, z)|\approx |\go (w, z)|
  \end{equation}
  where the constants implied in the inequality $\lesssim$ above, and the equivalence $\approx$, are independent 
  of $\epsilon$. 
\subsection{Special coordinate system}\label{SS:spec-coord}\quad
To obtain a better understanding of $\go$ and $\ge$ we introduce for each $w\in\bndry D$ a special coordinate system centered at $w$. We let $\nu_w$ denote the inner unit normal at $w$, so 
$\nu_w=-\nabla\r(w)/|\nabla\r (w)|$. We set $e_n = i\nu_w$, and take $\{e_1, \ldots, e_{n-1}, e_n\}$ to be an orthonormal basis of $\Cn$. Our coordinates of a point $z\in\Cn$ are then determined by
\begin{equation*}
z-w = \sum\limits_j z_je_j\, .
\end{equation*}
Note that then the coordinate $z_n = x_n + i y_n$ is intrinsically determined,
 as well as the length of the orthogonal complement, 
\begin{equation*}
|z'|=\left(\,\sum\limits_{j=1}^{n-1}|z_j|^2
\right)^{\!\!1/2}.
\end{equation*}
Using the fact that $2\,\dee\r (w) = -i|\nabla\r (w)|dw_n$, 
we see that 
\begin{equation*}
\langle\dee\r (w), w-z\rangle = -\frac{i}{2}|\nabla\r (w)|\, z_n
\end{equation*}
(see also \cite{LS-4}).
Looking back at \eqref{E:2.3} we then have
\begin{equation}\label{E:2.6}
\ge(w, z) = -ic_w\, z_n + Q(z)
\end{equation}
where $c_w = |\nabla\r (w)|/2$, and $Q(z) = Q^\epsilon_w(z)$ is a homogeneous quadratic polynomial in
$z_1,\ldots,z_n$. The only property of $Q$ that we need to know is that $|Q(z)|\leq c|z|^2$, with a 
constant $c$ independent of $\epsilon$ and $w$. 

The main estimates for $\ge (w, z)$ (and hence also for $\go (w, z)$) are contained in the following 
\begin{Prop}\label{P:1}
For each $w\in\bndry D$, we have
\begin{itemize}
\item[(i)]\quad $|\ge (w, z)| \approx |x_n| + |w-z|^2 + |\r (z)|,\quad \mbox{for}\ z\in\bar D$.
\end{itemize}
In particular, we have
\begin{itemize}
\item[(ii)]\quad $|\ge (w, z)| \approx |x_n| + |z'|^2, \quad \mbox{for}\ z\in \bndry D$.
\end{itemize}
The constants implicit in these equivalences are independent of $\epsilon$, $w$ and $z$.
\end{Prop}
\begin{proof}
We begin by observing that $z=w$ we have
\begin{equation*}
\frac{\dee\r}{\dee z_j} =0,  j=1,\ldots, n-1; 
\quad 
\frac{\dee\r}{\dee x_n} =0;
\quad
\frac{\dee\r}{\dee y_n} = -|\nabla\r (w)|.
\end{equation*}
Thus, by Taylor's theorem,
\begin{equation}\label{E:2.7}
\r (z) = -|\nabla\r (w)|\,y_n + O(|z-w|^2),\quad \mbox{and therefore}
\end{equation}
\begin{equation*}
|y_n|\lesssim |\r (z)| + |w-z|^2.
\end{equation*}
Combining this with \eqref{E:2.6} and the comments thereafter gives
\begin{equation*}
|\Re \ge (w, z)|\lesssim |\r (z)| +|w-z|^2.
\end{equation*}
However when $z\in\bar D$ we have $-\r (z)\geq 0$, and so \eqref{E:2.4} grants the opposite
inequality, which yields
\begin{equation}\label{E:2.8}
\Re\ge (w, z)| \approx |\r (z)| + |w-z|^2.
\end{equation}
Moreover \eqref{E:2.6} immediately gives 
\begin{equation*}
\Im \ge (w, z) = -c_w\,x_n + O(|w-z|^2),\quad \mbox{which leads to}
\end{equation*}
\begin{equation}\label{E:2.9}
\left\{\begin{array}{rcl}
|\Im\ge (w, z)|&\lesssim&|x_n| + |w-z|^2,\quad\mbox{and}\\
\\
|x_n|&\lesssim& |\Im\ge (w, z)|+|w-z|^2.
\end{array}
\right.
\end{equation}
A combination of \eqref{E:2.9} with \eqref{E:2.8} then implies conclusion (i). 

The second conclusion 
follows from the first because $z\in\bndry D$ implies $\r (z)=0$, while $|w-z|^2 = |z_n|^2 + |z'|^2 = x_n^2 + y_n^2 + |z'|^2$, and
$x_n^2\approx |x_n|$ if $x_n$ is bounded, while by \eqref{E:2.7}, $|y_n|\lesssim |w-z|^2$, if $z\in\bndry D$.
\end{proof}
There are two consequences that can be drawn from Proposition \ref{P:1}. First, if $z$ and $w$ are both 
in $\bndry D$ then
\begin{equation}\label{E:2.10}
|\ge (w, z)|\approx |\Im\langle\dee\r (w), w-z\rangle| + |w-z|^2.
\end{equation}
In fact, by the above
\begin{equation*}
\frac{\dee\r (w)}{\dee z_n} = \frac{i}{2} |\nabla\r (w)|,\quad \mbox{and hence}\ 
\langle\dee\r (w), w-z\rangle=-ic_wz_n,
\end{equation*}
with $c_n>0$ as in \eqref{E:2.6}; 
so $\Im \langle\dee\r (w), w-z\rangle = -c_w x_n$. Thus \eqref{E:2.10} is a consequence of the first conclusion of Proposition \ref{P:1}. 

Our next assertion is the analogue of \cite[Lemma 4.3]{LS-4}. For $z\in\bndry D$ we write $z^\delta = z+ \delta\nu_z$, where $\nu_z$ is the inward unit normal at $z$.
\begin{Cor}\label{C:2}
For $\delta>0$ sufficiently small, we have
\begin{equation*}
|\ge (w, z^\delta)|\approx |\ge (w, z)| +\delta,\quad \mbox{for}\ w, z\in\bndry D,\ \mbox{and}\ z^\delta
 \mbox{as above}.
\end{equation*}
\end{Cor}
\begin{proof}
It suffices to prove this when $|w-z|\leq c_1$, where $c_1$ is a small positive constant to be chosen below.
The result for $|w-z|\geq c_1$ is a trivial consequence of conclusion (i) in Proposition \ref{P:1}. 

Now let $z_n=x_n + iy_n$ be the $n$-th coordinate of $z$ in the coordinate system centered at $w$ that was chosen earlier, and $z_n^\delta = x_n^\delta + iy_n^\delta$ be the corresponding coordinate of $z^\delta = 
z + \delta\nu_z$. Then, letting $(\cdot, \cdot)$ denote the hermitian inner product in $\Cn$, we have
\begin{equation*}
z_n = (z-w, e_n),\ z_n^\delta = (z^\delta -w, e_n),\ \mbox{and therefore}
\end{equation*}
\begin{equation*}
z_n^\delta = z_n +\delta (\nu_z, e_n) = z_n +\delta (\nu_w, e_n) + O(\delta |z-w|),\ 
\mbox{while}\ (\nu_w, e_n) = i c_w.
\end{equation*}
 Hence 
\begin{equation}\label{E:2.12}
x_n^\delta = x_n + O(\delta|z-w|).
\end{equation}
Also, by Taylor's theorem $\r (z^\delta) = -|\nabla\r (z)| \delta + O(\delta^2)\ \mbox{as}\ \delta\to 0$ since
$\r(z)=0$ and $(\nu_z, \nabla\r (z))= -|\nabla\r (z)|$. Thus
\begin{equation}\label{E:2.13}
|\r (z^\delta)|\approx \delta \quad \mbox{for small}\ \delta >0.
\end{equation}
Finally, we note that $|z^\delta -w|^2 = |z-w|^2 + O(\delta|z-w| +\delta^2)$, and applying this along with
\eqref{E:2.12} and \eqref{E:2.13} via conclusion (i), we see that
\begin{equation*}
|\ge (w, z^\delta)| \approx |x_n^\delta| +|z^\delta -w|^2 +|\r (z^\delta)| 
\approx |x_n| + |z-w|^2 +\delta + O(\delta |z-w|+\delta^2).
\end{equation*}
Now we merely need to take $c_1$ and $\delta$ sufficiently small to absorb the ``$O$'' term above into 
$\delta$, to get
\begin{equation*}
|\ge (w, z)|\approx |x_n| + |w-z|^2 +\delta\,,
\end{equation*}
and thus using conclusion (i) again proves the corollary.
\end{proof}
Note that in view of \eqref{E:2.5} the conclusions of Proposition \ref{P:1} and Corollary \ref{C:2} hold for $\go$ as well as $\ge$.
\subsection{Geometry of the boundary of $D$}\quad 
We define the function $\d (w, z)$ by
\begin{equation*}
\d (w, z) = |\go (w, z)|^{1/2}.
\end{equation*}
Note that by \eqref{E:2.5} we also have
\begin{equation}\label{E:eg-est}
\d (w, z) \approx |\ge (w, z)|^{1/2}\quad\mbox{for all}\ \ \epsilon.
\end{equation}
For $w, z\in\bndry D$, $\d (w, z)$ has the properties of a quasi-distance, namely
\begin{Prop}\label{P:2} 
For $\d(w, z)$ defined as above, we have
\medskip

\begin{itemize}
\item[(a)]\quad $\d(w, z)\geq 0$, and $\d(w, z)=0$ only when $w=z$.
\medskip

\item[(b)]\quad $\d(w, z)\approx \d(z, w)$
\medskip

\item[(c)]\quad $\d(w, z)\lesssim \d(w, \zeta) + \d(\zeta, z)$
\medskip

\end{itemize}
whenever $w, \zeta, z\in\bndry D$.
\end{Prop}
\begin{proof}
Conclusion (a) is obvious from \eqref{E:2.10} and the definition of $\d(w, z)$. 
Next, observe that 
\begin{equation*}
|\Im \langle\dee\rho (w), w-z\rangle| = |\Im \langle\dee\rho (z), z-w \rangle| + O(|w-z|^2).
\end{equation*}
Therefore by \eqref{E:2.10}, we have that $|\ge (w, z)|\approx |\ge (z, w)|$, which implies conclusion (b). 
Finally, note that
\begin{equation*}
\langle\dee\r (w), w-z\rangle - \langle\dee\r (w), w-\zeta\rangle
= \langle\dee\r (w), \zeta - z\rangle
\end{equation*}
and the latter equals $\langle\dee\r (\zeta), \zeta - z\rangle + O(|\zeta -z|\, |w-\zeta|)$ since we have that
$\dee\r (w) -\dee\r (\zeta) = O(|w-\zeta|)$.
It follows that $|\Im\langle\dee\r (w),  w-z\rangle|$ is bounded above by
\begin{equation*}
|\Im\langle\dee\r (w),  w-\zeta\rangle| + |\Im\langle\dee\r (\zeta),  \zeta -z\rangle| +
O(|w-\zeta|^2 + |z-\zeta|^2).
\end{equation*}
Also, $|w-z|^2\lesssim |w-\zeta|^2 + |\zeta -z|^2$. From these observations and \eqref{E:2.10} we obtain
\begin{equation*}
|\go(w, z)|\lesssim |\go (w, \zeta)| + |\go (\zeta, z)|,
\end{equation*}
which grants conclusion (c).
\end{proof}
A final, simple observation about the quasi-distance $\d$ is that
\begin{equation}\label{E:2.14}
|w-z|\lesssim \d(w, z)\lesssim |w-z|^{1/2},\quad w, z\in\bndry D,
\end{equation}
which follows immediately from \eqref{E:2.10} via the Cauchy-Schwartz inequality, and the definition of $\d$.
\medskip

At this stage it is worth recording the following facts proved in the same
spirit as the proof of Proposition \ref{P:2}. 

For $A(w, z)$ equal to $\ge (w, z)$ 
we have
\begin{equation}\label{E:starp}
|A (w, z) -A(w', z)|\ \leq\ c_\epsilon\big( \d (w, w')^2 +\d (w, w')\,\d(w, z)\big)\, 
\end{equation}
and for $A(w, z)$ equal to either $\ge (z, w)$ or $\mbox{Im}\langle\dee\rho(w), w-z\rangle$
\begin{equation}\label{E:star}
|A (w, z) -A (w', z)|\ \lesssim\ \d(w, w')^2 + \d(w, w')\,\d(w, z)
\end{equation}
where the implicit constant in the second of these inequalities does not depend on $\epsilon$.
We prove the first inequality; the inequality for $A(w, z) =\ge (z, w)$ or $\mbox{Im}\langle\dee\rho(w), w-z\rangle$ will follow by a similar argument.  We see from \eqref{E:2.3} that 
\begin{equation*}
\ge (w, z) = \langle\dee\rho (w), w-z\rangle + Q_w(w-z), \quad \mbox{if}\quad |w-z|\leq \mu/2
\end{equation*}
where $Q_w(u)$ is a quadratic form in $u$. Thus we may split the difference $\ge (w, z) - \ge (w', z)$ as the sum of two terms: $I+II$, where 
$$
I= \langle\dee\rho (w), w-z\rangle - \langle\dee\rho (w'), w'-z\rangle
$$
and 
$$
II= Q_w(w-z) - Q_{w'}(w'-z).
$$
Now 
$$
I = \langle\dee\rho (w), w-w'\rangle + \langle \dee\rho (w)-\dee\rho (w'), w'-z\rangle =
$$
$$
=\ge(w, w') - Q_w(w-w') + \langle \dee\rho (w)-\dee\rho (w'), w'-z\rangle\,  .
$$
\smallskip

Since $|\ge (w, w')|\approx \d(w, w')^2$, the identity above and \eqref{E:2.14} give
$$
|I|\lesssim \d(w, w')^2 +\d(w, w')\, \d(w, z).
$$
On the other hand
$$
II\ \leq\ \big|Q_w(w-z) -Q_w(w'-z)\big|\ +\  \big|Q_w(w'-z) - Q_{w'}(w'-z)\big|\, ,
$$
with
$$|Q_w(w-z) -Q_w(w'-z)|\lesssim \d(w, w')^2 +\d(w, w')\,\d(w, z)\, ,$$
whereas
\begin{equation*}
|Q_w(w'-z) - Q_{w'}(w'-z)|\
\leq c_\epsilon\,|w-w'|\,|w'-z|^2
\end{equation*}
where 
$$
c_\epsilon = \sup\limits_
{\stackrel{w\in\bndry D}{1\leq j, k\leq n}}
|\nabla \tau^\epsilon_{j, k} (w)|,
$$
and this grants
$$
|Q_w(w'-z) - Q_{w'}(w'-z)|\ \leq\ c_\epsilon \big(\d(w, w')^2 +\d(w, w')\,\d(w, z)\big)\, .
$$
This proves \eqref{E:starp}; the proof of \eqref{E:star} is similar but does not involve the bound $c_\epsilon$.

\medskip

We next introduce the Leray-Levi measure $d\l$ defined on $\bndry D$. The proofs of the assertions in the rest of this section follow closely those given in a broadly parallel situation in \cite[Section 3.4]{LS-4} and so details
will be omitted. The Leray-Levi measure $d\l$ on $\bndry D$ is defined by the linear functional
\begin{equation*}
f\mapsto \int\limits_{\bndry D}\!\! f(w)\,d\l (w) =
\frac{1}{(2\pi i)^n}\int\limits_{\bndry D}\!\! f(w)\, j^*\!\left(\dee\rho\!\wedge\!(\deebar\dee\r)^{n-1}\right)(w)\, ,
\end{equation*}
with $\r$ our defining function, and where $(\deebar\dee\r)^{n-1}$ is the $(n-1)$-fold wedge product of
$\deebar\dee\rho$, and with $j^*$  denoting the pull-back under the inclusion
\begin{equation*}
j: \bndry D \hookrightarrow \Cn\, .
\end{equation*}
Then one has
\begin{equation}\label{E:2.15}
d\l (w) \!=\! (2\pi i)^{-n}\! 
j^*\!\left(\dee\rho\!\wedge\!(\deebar\dee\r)^{n-1}\right)\!\!(w)\!=\!
\Lambda (w) d\sigma (w), 
\end{equation}
where $d\sigma$ is the induced Lebesgue measure,  and $\Lambda (w)$ is a continuous function such that
$$
\ c_1\leq \Lambda(w)\leq c_2\, ,\quad w\in \bndry D\, ,
$$
 with $c_1$ and $c_2$ two positive constants. In fact
\begin{equation*}
\Lambda (w) = (n-1)!(4\pi)^{-n}|\det \r (w)|\ |\nabla\r (w)|,\quad w\in\bndry D,
\end{equation*}
where $\det \r (w)$ is the determinant of the $(n-1)\times (n-1)$ matrix
\begin{equation*}
\left\{\frac{\dee^2\r}{\dee z_j\deebar z_k}\right\}\!\bigg|_{z=w}\, ,\quad 1\leq j, k\leq n-1,
\end{equation*}
computed in the coordinate system $(z_1, \ldots, z_n)$ centered at $w$ that was introduced above.
(See \cite[Lemma VII.3.9]{Ra}.) The $C^2$ character of $\r$ together with its strict plurisubharmonicity then 
establishes \eqref{E:2.15}. The particular relevance of the Leray-Levi measure will become apparent
when we consider adjoints of our operators.

Our next assertions concern the boundary balls $\{\B_r (w)\}$ determined via the quasi-distance $\d$ and their measures. We define
\begin{equation*}
\B_r(w) = \{z\in\bndry D\ :\ \d(w, z)<r\},\quad \mbox{where}\ w\in\bndry D.
\end{equation*}
We also consider the ``box'' 
\begin{equation*}
\Bt_r(w) =\{z\in\bndry D \ :\ |x_n|<r^2,\ |z'|<r\},\quad w\in\bndry D.
\end{equation*}
We then have the equivalence $\B_r(w) \approx \Bt_r(w)$, which means
\begin{equation*}
\Bt_{c_r}(w)\subset \B_r(w)\subset\Bt_{c_2r}(w),
\end{equation*}
for two positive constants $c_1$ and $c_2$ that are independent of $r>0$ and $w\in\bndry D$. These inclusions 
follow directly from conclusion (ii) in Proposition \ref{P:1} and the fact that
\begin{equation*}
|\go (w, z)|\approx |\ge (w, z)|\approx \d(w, z)^2\, .
\end{equation*}

Using the arguments set down in \cite[Section 3.5]{LS-4}, one can show that the equivalence of $\B_r(w)$ with $\Bt_r(w)$ implies that
\begin{equation*}
\l (\B_r(w))\approx r^{2n},\quad 0\leq r\leq 1,
\end{equation*}
and hence also $\sigma (\B_r(w))\approx r^{2n},\quad 0\leq r\leq 1$. 

As a final consequence (as is shown in \cite{LS-4}) we have 
\begin{equation}\label{E:2.16}
\int\limits_{w\in \B_r(z)}\!\!\!\!\!\!\!\! \d(w, z)^{-2n +\beta} d\l(w) \leq c_\beta\, r^\beta; 
\!\!\!\!\!
\int\limits_{w\notin \B_r(z)}\!\!\!\!\!\!\!\!\! \d(w, z)^{-2n -\beta} d\l(w) \leq c_\beta\, r^{-\beta}
 \end{equation}
for $0<r<1$ and $\beta >0$. When $\beta =0$ we can assert that
\begin{equation}\label{E:2.16a}
\int\limits_{w\notin \B_r(z)}\!\!\!\!\! \d(w, z)^{-2n}\, d\l (w)\ \leq\ c\, \log(1/r),\quad \mbox{for}\ \ 0< r<1/2.
\end{equation}
\section{A family of Cauchy integrals: definition, correction and initial properties}\label{S:3}
Here we define the Cauchy integrals $\{\Cine\}_\epsilon$ (determined by the the 
denominators $\{\ge\}_\epsilon$) and study their
properties when $\epsilon$ is kept fixed; for convenience of notation we will henceforth drop explicit reference
 to $\epsilon$ and will resume doing so in Part II, when the dependence on $\epsilon$ will again be relevant
  (in fact crucial).
 Thus, for the time being we will write $\g$ for $\ge$; $\Cin$ for $\Cine$, and so forth.
\subsection{A Cauchy-Fantappi\'e integral}\label{SS:initial}\quad 
Our Cauchy integral will be defined as the sum of two operators. The first, $\CinI$, is a Cauchy-Fantappi\'e integral. To describe it we first isolate a 1-form $G$ closely related to the denominator $\g$. 

We set
\begin{equation}\label{E:3.1}
G (w, z) = 
\end{equation}
\begin{equation*}
=\chi\bigg[\dee\r (w) -\frac12\sum\limits_{j, k} \tau^\epsilon_{j, k}(w)\, (w_j-z_j) \, dw_k\bigg] +
(1-\chi)\sum\limits_k(\bar w_k - \bar z_k)\, dw_k\, .
\end{equation*}
As a result,
\begin{equation*}
\g (w, z) = \langle G (w, z), w-z\rangle.
\end{equation*}
We next normalize $G$ and set
\begin{equation*}
\G (w, z) = \frac{G(w, z)}{\g (w, z)},\quad \mbox{for}\ w\in\bndry D, \ z\in D.
\end{equation*}
Then $\G$ is a ``generating form'', namely $\langle\G(w, z), w-z\rangle = 1$ for any $z\in D$ and for any $w$ 
in a neighborhood of $\bndry D$, see \cite[Lemma 6, Section 7]{LS-3}. Note that $\G$ is a form of type $(1, 0)$ in $w$, with coefficients that are $C^1$ in $w$ and $C^\infty$ in $z$. 

The Cauchy-Fantappi\'e integral $\CinI$ is defined as
\begin{equation*}
\CinI (f) (z) = \frac{1}{(2\pi i)^n}\!\!
\int\limits_{w\in\bndry D}\!\!\!\!
f(w)\, j^*(\G \!\wedge\! (\deebar \G)^{n-1})(w, z),\quad z\in D.
\end{equation*}
Here $f$ is an integrable function on $\bndry D$ and, as before, $j: \bndry D \hookrightarrow\C^n$.

Note that
\begin{equation*}
\deebar\G = -\frac{1}{\g^2}(\deebar\g)\wedge G + \frac{1}{\g}\, \deebar G=
-\frac{1}{\g} (\deebar \g)\wedge \G + \frac{1}{\g} \, \deebar G,
\end{equation*}
and since $\G\wedge\G =0$, it follows that
\begin{equation}\label{E:3.2}
\CinI (f) (z) = \int\limits_{w\in\bndry D}\!\!\! C^1(w, z) f(w),\quad z\in D,
\end{equation}
where
\begin{equation*}
C^1(w, z)=
 \frac{1}{(2\pi i)^n}\,j^*\!\left(\frac{G\!\wedge\! (\deebar G)^{n-1})(w, z)}{\g (w, z)^n}\right).
 \end{equation*}
\begin{Prop}\label{P:3}
Suppose that $F$ is continuous in $\bar D$ and holomorphic in $D$. Let 
\begin{equation*}
f= F\bigg|_{\bndry D}\, .
\end{equation*}
Then
\begin{equation*}
\CinI (f) (z) = F(z),\quad z\in D.
\end{equation*}
\end{Prop}
\begin{proof}
This proposition is a restatement of \cite[Proposition 4, Section 7]{LS-3} (see also 
\cite[Proposition 2, Section 5]{LS-3} and \cite[Lemma 6, Section 7]{LS-3}).
\end{proof}

\subsection {Correction Operator}\label{SS:correction}\quad
While the Cauchy-Fantappi\'e integral $\CinI$ reproduces holomorphic functions, $\CinI (f)$ is not holomorphic for general $f$. To achieve this we correct $\CinI$ by the solution of a $\deebar$-problem.
Here we use the presentation of this idea as it appears in \cite[Section 8]{LS-3}, where further details can be found; earlier versions are in \cite{KS-2} and \cite{Ra}.

There is a $C^\infty$-smooth, strongly pseudo-convex domain $\Omega$ that contains $\bar D$ with the property that  
 \begin{equation*}
 H(w, z) :=
\left\{\begin{array}{rcl}
-\deebar_z\big(\G\!\wedge\!(\deebar_w\G)^{n-1}\big),&\mbox{for}&
z\in\Omega\setminus\{|z-w|<\mu/2\}\\
\\
0,& \mbox{for}& |z-w|<\mu/2
\end{array}
\right.
\end{equation*}
is smooth in $z\in\Omega$ and continuous in $w\in \bndry D$, and satisfies the compatibility condition
 $\deebar_z H(w, z) =0$ whenever $z\in \Omega$ and $w\in\bndry D$. 
 So if we consider the solution of the $\deebar$-problem 
 \begin{equation*}
 \deebar_z\CkerII (w, z) = H(w, z),\quad z\in \Omega\, ,
 \end{equation*}
we can write $\CkerII (w, z) = \mathscr S_z (H(w, \cdot))$ for the corresponding normal solution operator $\mathscr S_z$, as given in e.g., 
\cite{CS}, \cite{FK}. Then $\CkerII (w, z)$ is an $(n, n-1)$-form in $w$, whose coefficients are of class $C^1$ in $w$ and depend smoothly on $z$. In particular, $\CkerII (w, z)$ is bounded on 
$\bndry D\times \bar D$, and so
\begin{equation}\label{E:C-two-bdd}
\sup\limits_{(w, z)\in \bndry D\times \bar D}|\CkerII (w, z)| \lesssim 1\, .
\end{equation}
We then define
\begin{equation}\label{E:3.3a}
\CinII (f) (z) =\int\limits_{\bndry D}\!\!\CkerII (w, z) f(w)\, 
\end{equation}
and write
\begin{equation*}
\Cin (f) (z) =\int\limits_{\bndry D}\!\!\Cker (w, z) f(w)\quad
\end{equation*}
where
 \begin{equation*}
 \Cker (w, z)= \CkerI(w, z) +\CkerII (w, z),\quad \mbox{and so}\quad  \Cin = \CinI + \CinII\, .
\end{equation*}
One has  as a result
\begin{Prop}\label{P:4}\quad
\begin{itemize}
\item[(1)]\quad Whenever $f$ is integrable, $\Cin (f) (z)$ is holomorphic for $z\in D$.
\item[]
\item[(2)]\quad  If $F$ is continuous in $\bar D$ and holomorphic in $D$ and 
\begin{equation*}
f=F\bigg|_{\bndry D}\, ,
\end{equation*}
\quad then\ \  $\Cin (f) (z) = F(z), \ z\in D.$
\end{itemize}
\end{Prop}
It is also useful to have the additional regularity of $\Cin (f)$ when $f$ is H\"older in
the sense of the quasi-distance $\d$, i.e. it satisfies
\begin{equation}\label{E:3.4}
|f(w_1)-f(w_2)|\lesssim \d(w_1, w_2)^\alpha\quad \mbox{for some}\ 0<\alpha\leq 1.
\end{equation}
\begin{Prop}\label{P:5}
If $f$ satisfies the H\"older-type condition \eqref{E:3.4} then $\Cin (f)$ extends to a continuous function
on $\bar D$.
\end{Prop}
\begin{proof}
Since $\Cin = \CinI + \CinII$ and the kernel of $\CinII$ is continuous in $\bar D$, $\CinII (f)$ is automatically
continuous there, and we are reduced to considering $\CinI (f)$.

Given the smoothness of $\G (w, z)$ and $\deebar_w\g(w, z)$ for $z\in D$, it suffices to prove the 
continuity of $\CinI (f) (z)$ for $z$ in $\bar D$ and 
close to
  $\bndry D$. To do this we set 
\begin{equation*}
F^\delta (z) = \CinI (f)(z^\delta), \quad \mbox{with}\ z^\delta = z+\delta \nu_z
\end{equation*}
with $\nu_z$ and $\delta$ as in Corollary \ref{C:2}. It will suffice to see that the functions $F^\delta $ 
converge uniformly on $\bndry D$ as $\delta\to 0$. In fact one can assert that
\begin{equation}\label{E:3.5}
\sup\limits_{z\in\bndry D} |F^{\delta_1}(z) - F^{\delta_2}(z)|\lesssim \max\{\delta_1, \delta_2\}^{\alpha/2}.
\end{equation}
This can be proved as follows. With $\CkerI (w, z)$ the kernel of the operator $\CinI$ we have
\begin{equation}\label{E:3.6}
F^{\delta_1}(z) - F^{\delta_2}(z) =
\int\limits_{w\in\bndry D}\!\!\!
\bigg(\CkerI (w, z^{\delta_1}) -\CkerI (w, z^{\delta_2})\bigg)(f(w) - f(z))
\end{equation}
because $\CinI (1) =1$ by Proposition \ref{P:3}.

Now suppose $\delta_1\leq \delta_2$, then one has the following estimate
\begin{equation}\label{E:3.7}
\left|\CkerI (w, z^{\delta_1}) -\CkerI (w, z^{\delta_2})\right| \lesssim
\min\left\{\frac{1}{\d(w, z)^{2n}}, \  \frac{\delta_2}{\d(w, z)^{2n+2}}\right\}\, .
\end{equation}
Indeed, $\CkerI (w, z)=j^*\left((2\pi i\, \g (w, z))^{-2n} G\wedge (\deebar_w G)^{n-1}\right)$. Looking back at the definition
of $G(w, z)$, see \eqref{E:3.1}, we see that $G$ and $\deebar_w G$ are bounded. So the inequality:
$\left|\CkerI (w, z^{\delta_1}) -\CkerI (w, z^{\delta_2})\right| \lesssim \d(w, z)^{-2n}$ follows from Corollary
 \ref{C:2}, if we take into account that $|\g(w, z)|\approx \d(w, z)^2$. 
 
 Since $\nabla_{\!z} G$ and $\nabla_{\! z}\deebar_w G$ are also bounded, Corollary \ref{C:2} again grants
 $$\left|\CkerI (w, z^{\delta_1}) -\CkerI (w, z^{\delta_2})\right| \lesssim |\delta_1-\delta_2|\,|\g(w, z)|^{-n-1}\lesssim
 \delta_2\,\d(w, z)^{-2n-2},$$ since $\delta_2>\delta_1$. Hence the estimate \eqref{E:3.7} is established.
 
 We now break the integration in \eqref{E:3.6} into two parts: where $\d(w, z)\leq \delta_2^{1/2}$, and
  $\d(w, z)>\delta_2^{1/2}$. Then since $|f(w_1)-f(w_2)|\lesssim\d(w, z)^\alpha$, the integral over the first part
  is bounded by a multiple of 
  $$
  \int\limits_{\d(w, z)\leq \delta_2^{1/2}}\!\!\!\!\!\!\!\!\!\!\d(w, z)^{-2n+\alpha}\,d\l(w)\, .
  $$
  Similarly the integral over the second part is bounded by
   $$
 \delta_2\!\! \!\!\!\!\!\!\!\!\int\limits_{\d(w, z)> \delta_2^{1/2}}\!\!\!\!\!\!\!\!\!\!\d(w, z)^{-2n-2 +\alpha}\,d\l(w)\, .
  $$
Both integrals are $O(\delta_2^{\alpha/2})$ in view of \eqref{E:2.16} with $r=\delta_2^{1/2}$.
Since we took $\delta_2\geq \delta_1$, this proves \eqref{E:3.5}, and the proposition is established.
\end{proof}
\section{A Cauchy transform; $L^p$-boundedness}\label{S:4}
Proposition \ref{P:5} allows us to define the  ``Cauchy transform'' $\Ctr$. It is a linear operator, initially defined on functions satisfying \eqref{E:3.4} by
\begin{equation}\label{E:4.1}
\Ctr (f) = \Cin (f)\big|_{\bndry D}
\end{equation}
(We recall that $\Cin$ and thus $\Ctr$ depend on a parameter $\epsilon$, which here is kept fixed and hence omitted from the notations.) 
\medskip

It is worthwhile to point out the following formula for $\Ctr (f)$
\begin{equation}\label{E:Ctr-rep}
\Ctr (f)(z) = f(z)\ + \int\limits_{w\in\bd D}\!\!\!
C(w, z) \,[f(w)-f(z)]\, ,\quad z\in\bd D\, ,
\end{equation}
which is proved by considering the identity
\begin{equation*}
\Cin (f)(z^\delta) = f(z)\ + \!\int\limits_{w\in\bd D}\!\!\!\!
C(w, z^\delta) \,[f(w)-f(z)]
\end{equation*}
where as before, $z^\delta = z+\delta \nu_z$, and letting $\delta\to 0$. (We remark
that the integral in the expression above is absolutely convergent in view of the fact 
that $f$ satisfies the H\"older-like condition \eqref{E:3.4}.)
\medskip

Our principal result for $\Ctr$ is as follows.
\begin{Thm}\label{T:1}
The operator $\Ctr$ initially defined for functions satisfying \eqref{E:3.4} extends to a bounded linear transformation on $L^p(\bndry D, d\l)$, for $1<p<\infty$.
\end{Thm}
On account of the equivalence \eqref{E:2.15} this also gives boundedness in $L^p(\bndry D, d\sigma)$.

\subsection{Essential parts and remainders}\label{SS:dec-two}
 \quad At several steps of our analysis 
   we require a decomposition of the Cauchy integral $\Cin$ and the  Cauchy transform $\Ctr$ into an ``essential part'' plus an acceptable ``remainder'',
    and moreover such decomposition appears in several different forms.
   In fact, first we considered $\Cin = \CinI+\CinII$, with $\CinI$ a Cauchy-Fantappi\'e integral and $\CinII$ a correction term via $\deebar$ (the ``remainder''), and this led to the definition \eqref{E:4.1} of the Cauchy transform $\Ctr$. A new decomposition of $\Ctr$ is given immediately below;
     further decompositions of $\Ctr$ will be needed in Part II, when we study the Cauchy-Szeg\H o projection.
     
      It is worthwhile to point out
   that the remainders that appear in all such decompositions will always be less singular than the corresponding essential parts,
     in the sense that their kernels (hereby generically denoted $R(w,z)$)  are
   easily seen to be controlled via the improved bound
      \begin{equation}\label{E:remainder-bound-a}
 |R(w, z)|\lesssim   \d (w, z)^{-2n+1}
    \end{equation} 
    for any $w, z\in \bndry D$, and satisfy the improved difference condition
     \begin{equation}\label{E:remainder-bound-b}
 |R(w, z)-R(w, z')|\lesssim   \frac{\d(z, z')}{\d (w, z)^{2n}}
    \end{equation} 
whenever  $\d(w, z)\geq c\,\d(z, z')$
 for an appropriate large constant $c$. 
 
 By contrast, the kernels of the essential parts 
       will exhibit no such improvements and will in fact retain the same singularities as the original transform $\Ctr$, see for instance \eqref{E:4.13} below.
    
  \subsection{A new decomposition of $\Ctr$}
  We begin by making a decomposition
   of the Cauchy transform $\Ctr$ that
     will eventually allow us to study its ``adjoint'' on $L^2(\bndry D, d\l)$.  
   Here we take the numerator of the Cauchy-Fantappi\'e integral, essentially
    $j^*\left(G\wedge(\deebar G)^{(n-1)}\right)/(2\pi i)^n$, see \eqref{E:3.2}, and replace it by 
    $j^*(\dee\rho\wedge(\deebar\dee\rho)^{n-1)})/(2\pi i)^n=d\l$, the Leray-Levi measure. 
    So we define
     $\Cins$ by
     \begin{equation}\label{E:4.5}
     \Cins (f) =\!\!\int\limits_{w\in \bndry D}\!\frac{f(w)}{\g(w, z)^n}\, d\l (w)\,,  \quad z\in D.
     \end{equation}
    Looking back at \eqref{E:3.1} we see that
    \begin{equation*}
    G = \dee\rho +\gamma_0 +\sum\limits_{k=0}^n (w_k-z_k) \gamma_k
    \end{equation*}
    where $\gamma_0$, $\gamma_k$ are 1-forms whose coefficients are smooth in $z$ and of 
    class $C^1$ in $w$, with $\gamma_k$ supported in $|w-z|\leq\mu$ and $\gamma_0$
    supported where $|w-z|\geq \mu/2$. 
    As a result
    \begin{equation}\label{E:4.6}
    \Cin (f) (z) = \Cins (f) (z) +
       \Rin (f) (z),\quad z\in D
    \end{equation}
    where the remainder
       $\Rin$ is given by
    \begin{equation}\label{E:4.3}
 \Rin f (z) = \int\limits_{\bndry D}\!\!
 \g (w, z)^{-n}j^*\!
 \big(\alpha_0(w, z) + \sum\limits_j\alpha_j(w, z)(w_j-z_j)\big)f(w)
 \end{equation}
 \begin{equation*}
  \!\!\!\!\!\!\!\!\!\!\!\!\!\!\!\!\!\!\!\!\!\!\!\!\!\!\!\!\!\!\!\!\!\!\!
  \!\!\!\!\!\!\!\!\!\!\!\!\!\!\!\!\!\!\!\!\!\!\!\!\!\!\!\!\!\!\!\!\!\!\!
  \!\!\!\!\!\!\!\!\!
   +\quad \CinII (f)(z).
 \end{equation*}
 Here $\alpha_0(w, z)$ and $\alpha_j(w, z)$ are $(2n-1)$-forms in $w$ whose coefficients are continuous in $w\in\bndry D$ and  smooth in $z\in\bar{D}$, with $\alpha_0 (w, z)$supported away from the diagonal $\{w=z\}$. Also, $\CinII (f)$ is the correction term \eqref{E:3.3a}.
  
 In view of \eqref{E:C-two-bdd} and of the fact that $|\g (w, z)|\approx \d(w, z)^2$,
 we see that the kernel of $\Rin$ is bounded by a multiple of $\d(w, z)^{-2n+1}$, see
 \eqref{E:remainder-bound-a}. Thus, 
 by
 \eqref{E:2.16} and Corollary \ref{C:2} the integral defining 
  $\Rin (f)$ converges absolutely and uniformly for $z\in D$ and hence extends to a continuous function on $\bar D$.

    With this and \eqref{E:4.1} we can define the operator $\Ctrs$ acting on $f$ that satisfies \eqref{E:3.4} by
    \begin{equation*}
    \Ctrs (f) = \Cins (f)\bigg|_{\bndry D}\, .
    \end{equation*}
    As a result we have a corresponding identity
    \begin{equation}\label{E:4.6a}
    \Ctr (f) (z) = \Ctrs (f) (z) +
        \Rtr (f) (z),\quad z\in \bndry D\, .
    \end{equation}
    We point out that
      \begin{equation}\label{E:4.6b}
    \Ctrs (f) (z) = \int\limits_{w\in\bndry D}\!\!\!
    \frac{f(w)}{g(w, z)^n}\,d\lambda (w)
    \end{equation}
    whenever $f$ satisfies the H\"older-like condition \eqref{E:3.4} and $z\in\bndry D$ is such that $f(z)=0$. This results by a passage to the limit in \eqref{E:4.5} where we use
    \eqref{E:4.1} and \eqref{E:Ctr-rep}.  The  remainder operator $\Rtr$ will be dealt with in Section 
    \ref{SS: C-tr-bdd}.
    
\subsection{$\Ctrs$ as a ``derivative''}\quad
A turning point in this analysis
is the realization that $\Ctrs (f)$ can be appropriately expressed
in terms of $df$. For this purpose we define two operators $\Ein$ and $\Rins$. Here the ``essential part'' $\Ein$ acts on continuous 1-forms $\mu$ on $\bndry D$, and the output $\Ein (\mu)$ is the continuous function in $\bar D$
given by 
\begin{equation*}
\Ein (\mu) (z) =c_n\!\!\!\!\int\limits_{w\in\bndry D}\!\!\!\!
\g(w, z)^{-n+1}\,\mu (w)\wedge j^*(\deebar \dee\rho (w, z))^{n-1}
\end{equation*}
with $c_n =1/[(n-1)(2\pi i)^n]$. The ``remainder'' $\Rins$ maps continuous functions on $\bndry D$ to continuous functions on $\bar{D}$  and is defined in a manner analogous to \eqref{E:4.3}.

 \begin{Prop}\label{P:6}
 If $f\in C^1(\bndry D)$, then
 \begin{equation}
 \label{E:4.4a}
 \Cins(f)(z) = \Ein(df)(z) +\Rins(f)(z),\quad \mbox{for}\ z\in D.
 \end{equation}
 \end{Prop}
 \noindent  An analogous result can be found in \cite[Proposition 5.1]{LS-4}. 
  As a consequence we have
  \begin{equation}\label{E:4.4}
  \Ctrs(f)(z) = \Etr(df)(z) +\Rtrs(f)(z),\quad \mbox{for}\ z\in \bndry D\, ,
  \end{equation}
  where $\Etr$ and $\Rtrs$ are defined as the corresponding limits on $\bndry D$ of $\Ein$ and $\Rins$.
  
 The simple integration lemma below will be used here and at several later occasions. Suppose $F$ and $\rho$ are a pair of functions on $\bndry D$ with $F$ of class $C^1$ and $\rho$ of class $C^2$.
 \begin{Lem}\label{L:9}
 \begin{equation*}
 \int\limits_{\bndry D}\!\! dF\wedge (\deebar\dee\rho)^{n-1}=0\, .
 \end{equation*}
 \end{Lem}
 If we make the stricter assumption that $\rho$ is of class $C^3$, then since $d(\deebar \dee\rho)=0$, the assertion of the lemma follows immediately from Stokes' theorem and the fact that $d(F\wedge (\deebar\dee\rho)^{n-1})= dF\wedge (\deebar\dee\rho)^{n-1}$. The case when $\rho$ is merely of class $C^2$ then 
 follows
  this by approximating $\rho$ in the $C^2$-norm by 
 $\{\rho_k\}_k$ with each $\rho_k$ of class $C^3$.
\medskip

 \noindent{\em Proof of Proposition \ref{P:6}.}\ 
 We fix $z$ in the interior of $D$;  our goal is to apply Lemma \ref{L:9} to
$$
F(w)=f(w)\cdot g(w, z)^{-n+1}
$$
(and $\rho$ the defining function of our domain). We claim that computing 
$dF\wedge (\dee\deebar \rho )^{n-1}$ for such an $F$ will give rise to several terms. One of these will be $-\g (w,z)^{-n} f(w) d\l (w)$ (whose integral over $\bndry D$ is precisely $-\Cins (f) (z)$), while the integrals of the other terms will be $\Ein (df) (z) +\Rins (f) (z)$. Specifically, we have
 \begin{equation*}
 dF = (-n+1)f(w)\cdot g(w, z)^{-n}d_wg(w, z) + df(w)\cdot g^{-n+1}(w, z),
 \end{equation*}
 and
  \begin{equation}\label{E:4.1.a}
  d_w\g = \dee\r (w) +\beta_0(w, z) + \sum\limits_j\beta_j(w, z)(w_j-z_j).
  \end{equation}
 Here $\beta_j$ and $\beta_0$ are 1-forms in $w$ that are smooth as $z$ varies over $\bar D$, and continuous as $w$ varies over $\bndry D$, with $\beta_0(w, z)$ supported away from the diagonal $\{w=z\}$. This leads to \eqref{E:4.4a},
thus proving the proposition.
 \qed

    \subsection{The virtual adjoint of $\Ctrs$}\quad Here it is crucial that we take the inner product, with respect to which the ``adjoint'' is defined, to be given by
    \begin{equation*}
    (f_1, f_2) = \int\limits_{\bndry D}\!\! f_1(w)\,\bar{f_2}(w)\, d\l (w)\, ,
    \end{equation*}
    with $d\l$ the Leray-Levi measure, and we will let $(\Ctrs)^*$ denote the adjoint of $\Ctrs$ with respect to such inner product. We have the following representation for $(\Ctrs)^*$.
    \begin{Prop}\label{P:7} 
    There is a linear operator $(\Ctrs)^{\!*}$, acting on functions that are H\"older in the sense of \eqref{E:3.4} for some $\alpha$, so that
    \begin{equation}\label{E:4.7a}
    (\Ctrs f_1, f_2) = (f_1, (\Ctrs)^{\!*} f_2)
    \end{equation}
    for any pair of such functions. Moreover
    \begin{equation}\label{E:4.7}
    (\Ctrs)^{\!*} (f) (z) = \int\limits_{w\in\bndry D}\!\!\!\bar{\g}(z, w)^{-n}f(w)\, d\l (w)
    \end{equation}
    for those $z\in\bndry D$ such $f(z)=0$.
    \end{Prop}
    (Notice that the integral above converges absolutely because when $z$ is outside the support of $f$ one has
    $|f(w)|\lesssim \d(w, z)^\alpha$.)
    \begin{proof}
    We consider the family of operators $\Ctrsd$, given by
    $$
    \Ctrsd (f) (z) = 
    \int\limits_{w\in\bndry D}\!\!\!\g(w, z^\delta)^{-n}f(w)\, d\l (w), \quad z\in\bndry D\, ,
    $$
    where $z^\delta = z+\delta\nu_z$ is defined as in Corollary \ref{C:2}.  
       
     Now $\g (w, z^\delta)^{-n}$ is uniformly bounded in $z$ and $w$ for each $\delta>0$ (see Corollary \ref{C:2}), and hence for each $\delta$ the operator $\Ctrsd$ is bounded in $L^2(\bndry D, d\l)$. Its (genuine) adjoint 
     $(\Ctrsd)^*$ is then given by
    \begin{equation}\label{E:4.7s}
    (\Ctrsd)^*(f)(z) = \int\limits_{w\in\bndry D}\!\!\!
    \bar g(z, w^\delta)^{-n}f(w)\, d\l (w)
    \end{equation}
    where again $w^\delta = w+\delta\nu_w$, and $(\Ctrsd (f_1), f_2) = (f_1, (\Ctrsd)^*(f_2))$, whenever
    $f_1$, $f_2\in L^2(\bndry D, d\l)$.  
    It will suffice to see that whenever $f$ is H\"older, $(\Ctrsd)^*(f)$ 
    converges uniformly, as $\delta\to 0$, to a limit, and to take $(\Ctrs)^*(f)$ to be such limit. The convergence will be shown using Proposition \ref{P:8} below, which is the special case when $f=1$; we give a separate proof as its full conclusion will be needed later.
       \begin{Prop}\label{P:8} With same notations as above, we have
    $$
    (\Ctrsd)^*(1) = \hda\, ,\quad \delta>0
    $$
    where $\{\hda\}_\delta$ are continuous functions on $\bndry D$. The $\hda$ converge uniformly to 
    a continuous function $\ha$ as $\delta\to 0$.
       \end{Prop}
    \begin{proof}
    The proof follows the spirit of Proposition \ref{P:6}, expressing in this case $(\Ctrsd)^*(f)$ with $f=1$ in terms of $df=0$, with an acceptable error term. Here we will apply Lemma \ref{L:9} with $F(w) = \bar{g}(z, w^\delta)^{-n-1}$.

  If   $|z-w|\leq \mu/2$ then
    \begin{equation*}
    d_w(\bar g(z, w^\delta)) = 
    \end{equation*}
    \begin{equation*}
    =\, d_w\!\left(\langle\deebar\rho (z), \bar z -\bar w^{\,\delta}\rangle\right) 
      \,  -\, \frac12\,d_w\!\left(\sum\limits_{j, k}\bar\tau^{\,\epsilon}_{jk}(z)(\bar z_j -\bar w^{\,\delta}_j)
    (\bar z_k -\bar w^{\,\delta}_k)
    \right)\, .
    \end{equation*}
    Now 
    $$
    \displaystyle{d_w\!\left(\langle\deebar\rho (z), \bar z -\bar w\rangle\right) = -\deebar\rho (w) +
    \sum\limits_j\!
    \left(\!
    \frac{\dee\rho (w)}{\dee\bar w_j} - \frac{\dee\rho (z)}{\dee\bar z_j}\, 
    \right)\!
    d\bar w_j\, .
    }$$
     Also, $w^\delta = w+\delta\nu_w$, and 
       if we write $\nu_w=(\nu_w^1,\ldots, \nu_w^n)$ then $d_w\nu_w$  are 1-forms with continuous coefficients. Altogether then
    \begin{equation*}
    d_w\!\left(\bar\g(z, w^\delta)\right) = -\deebar \rho (w)\ +\
    \end{equation*}
    \begin{equation*}
    +\beta_0^* + \sum\limits_{j}\beta^*_j(w, z) (\bar w_j-\bar z_j) +
    \sum\limits_{j}\beta^{**}_j(w, z)\!\left(\!
    \frac{\dee\rho (w)}{\dee\bar w_j} - \frac{\dee\rho (z)}{\dee\bar z_j}\, 
    \right)\! +\, \delta\,\gamma^\delta (z, w)\, .
    \end{equation*}
    Here $\beta_0^*$, $\beta^*_j$, $\beta^{**}_j$ and $\gamma^\delta$ are 1-forms in $w$ with
     coefficients that are $C^1$ in $z\in\bar{D}$ and  in $w\in \bndry D$; also $\gamma^\delta(z, w)$ is bounded as
     $\delta\to 0$, while $\beta_0^*(w, z)$ is supported away from the diagonal $\{w=z\}$. 
     
     Keeping in mind that
  $j^*d\rho =0$ (or its equivalent formulation: $j^*\dee\rho = -j^*\deebar\rho$), we obtain 
  from Lemma \ref{L:9} that
   $$
   \hda = (\Ctrsd)^*(1) = \int\limits_{\bndry D}\bar g(z, w^\delta)^{-n}\, d\l (w) \ =\ I_\delta^1\ +\ I^2_\delta
   $$
   with 
   \begin{eqnarray*}
   I_\delta^1 =\int\limits_{\bndry D} \bar \g(z, w^\delta)^{-n}j^*\big[a_0(w, z) +
   \sum\limits_{j=1}^na_j(w, z)(\bar w_j-\bar z_j) + b_j(w, z)\big(\frac{\dee\rho}{\dee\bar w_j}(w) - \frac{\dee\rho}{\dee \bar z_j}(z)\big)\big]
     \end{eqnarray*}
where       $a_0, a_j$ and $b_j$ are $(2n-1)$-forms in $w$, with coefficients that are continuous in $z\in D$ and $w\in\bndry D$, with $a_0(w, z)=0$ away from the diagonal $\{w=z\}$.  From this it is evident that
    $I^1_\delta$ converges uniformly as $\delta \to 0$ to a continuous function $I^1$. 

   Moreover
   $$
   I^2_\delta = \int\limits_{\bndry D} j^* (A_\delta(w, z))\, 
   $$
   where  $A_\delta (w, z)$ is a $(2n-1)$-form in $w$
     that
    satisfies the bound
   $$
   |A_\delta (w, z)|\lesssim \delta\,|\ge(w, z)|^{-n}\, ,
     $$
   and
   it follows by \eqref{E:2.16a} that $I^2_\delta = O(\delta\log{1/\delta})$ as $\delta\to 0$, 
   and hence this tends to zero uniformly in $z$, so that
   $$
   \hda (z)\to I^1(z) =: \ha (z)\quad \mbox{uniformly in }\ z,\quad \mbox{as}\ \delta\to 0.
   $$
   The proof of Proposition \ref{P:8} is concluded.
    \end{proof}
    We may now 
    complete
        the proof of Proposition \ref{P:7} . Turning back to $(\Ctrsd)^* (f)$ as expressed in
         \eqref{E:4.7s} we see that
    $$
    (\Ctrsd)^* (f) = \int\limits_{\bndry D}\!\!
    \bar{\g}(z, w)^{-n}[f(w)-f(z)]\, d\l(w) + \hda (z) f(z)\, ,
    $$
    since $(\Ctrsd)^* (1) = \hda $. Now if $|f(w)-f(z)|\lesssim \d (w, z)^\alpha, \alpha<1$, we clearly have the convergence
     of  $(\Ctrsd)^* (f)$ as $\delta\to 0$, to a limit which we call $(\Ctrs)^* (f)$ and which equals
     $$
     (\Ctrs)^* (f) = \int\limits_{\bndry D}\!\!
    \bar{\g}(z, w)^{-n}[f(w)-f(z)]\, d\l(w) + \ha (z) f(z). 
     $$
     With this the duality identities \eqref{E:4.7a} and \eqref{E:4.7} are established, proving Proposition \ref{P:7}.
         \end{proof}
    \subsection{$L^p$-boundedness of $\Ctr$}\label{SS: C-tr-bdd} We shall prove the  boundedness of $\Ctr$ by applying to $\Ctrs$ the $T(1)$ theorem in the form given in \cite{C-1}. To this end, in this section we go over the key hypotheses of this theorem,
      namely
    \begin{itemize}
    \item the cancellation properties of $\Ctrs$;
    \item the action of $\Ctrs$ on ``bump functions'';
    \item the difference estimates for the kernel of $\Ctrs$.
    \end{itemize}

    We first recall the cancellation properties of $\Ctrs$. We have
    \begin{equation*}
    \Ctrs (1) = \h\quad \mbox{and}\quad (\Ctrs)^*(1) =\ha
    \end{equation*}
    where both $\h$ and $\ha$ are continuous functions.
      This is what was proved for
    $(\Ctrs)^*(1)$ in Proposition \ref{P:8}. The proof works as well for $\Ctrs (1)$, using
     \eqref{E:4.4}, by which
    $\Ctrs (1)= \Rtrs (1)$.
       Now
        $\Rtrs (1)$ can be treated like $I^1_\delta$ in the proof of 
    Proposition \ref{P:8}.
    
     Next we come to the action of $\Ctrs$ on ``bump functions''. We fix $\alpha$, $0<\alpha<1$, and say that
    $f$ is a {\em normalized bump function associated with a boundary ball $\B_r (\wh)$} (with $\wh\in\bndry D$), if it is
    supported in $\B_r (\wh)$ and satisfies
    \begin{equation*}
    |f(z)|\leq 1,\quad |f(z)- f(z')|\leq \left(\frac{\d (z, z')}{r}\right)^{\!\alpha},\quad \mbox{for all}\quad z, z'\in \bndry D.
    \end{equation*}
  The main fact we use for these bump functions is a consequence of 
   Lemma \ref{L:general} below.
   
   Suppose $T$ is a linear transformation defined on functions satisfying 
  the H\"older-type regularity \eqref{E:3.4} for some $\alpha>0$, and mapping these to continuous functions on
  $\bndry D$. Assume further that $T$ has a kernel $K(w, z)$ so that
  \begin{equation}\label{E:T-repr-1}
  T(h)(z)= \int\limits_{\bndry D}\!\! K(w, z)h(w)\, d\lambda (w)
  \end{equation}
  holds whenever $h(z)=0$, with $|K(w, z)|\leq \d (w,z)^{-2n}$. Suppose in addition that
  \begin{equation}\label{E:star-bump}
  |T(f_0)(\wh)|\lesssim 1
  \end{equation}
  whenever $f_0$ is a $C^1$-smooth function supported in a ball $\B_r(\wh)$, and $|f_0|\leq 1$ and $|\nabla f_0|\leq 1/r^2$ on $\bndry D$. 
      \begin{Lem}\label{L:general}
     With the above  assumptions,
     the following holds whenever $f$ is a normalized bump function associated to $\B_r(\wh)$:

  \medskip
  
  (a)\quad $\sup\limits_{z\in\bndry D}|T(f)(z)|\lesssim 1$,\quad
   and
  \medskip
  
  (b)\quad $\| T(f)\|^2_{L^2(\bndry D, d\lambda)}\lesssim r^{2n}$.
   \end{Lem}
  \begin{proof}
  Let $f$ be a given normalized bump function associated with $\B_r(\wh)$,
  and let $\chi$ be  a non-negative $C^1$-smooth function on $\mathbb C$, so that $\chi(u+iv)=1$ for $|u+iv|\leq 1/2$, $\chi(u+iv)=0$ for $|u+iv|\geq 1$ and, furthermore, $|\nabla\chi (u+iv)|\leq 1/r^2$.  Set
  $$
  \tilde\chi_{r, \wh} ( w) = \chi\left(\frac{\mathrm{Im}\langle\dee\rho (\wh), \wh-w\rangle}{c\,r^2}\, +i\,\frac{|\wh-w|^2}{c\, r^2}\right)\ \  \mbox{for any}\ \ w\in\bndry D.
  $$
   If the constant $c$ is chosen sufficiently large,  it follows from \eqref{E:2.10} that $\tilde\chi_{r, \wh}(w) =0$ if $\d(\wh, w) \geq r$, while 
$\tilde\chi_{r, \wh}(w) =1$ whenever $\d(\wh, w)\leq c'r$ (with $c'$ another constant).
Now define
  \begin{equation*}
    f_0 (w) = f(\wh)\,\tilde\chi_{r, \wh}(w).
    \end{equation*}
  It is clear that $f_0$ satisfies all the requirements ensuring that 
  \eqref{E:star-bump} holds.
  We make the following four assertions:
  \medskip
  
  {\em (1)}\quad $\displaystyle{|(f-f_0)(w)|\lesssim 
  \left(\frac{\d(w, \wh)}{r}\right)^{\!\alpha}}$;
  \smallskip
  
  {\em (2)}\quad $|T(f)(\wh)|\lesssim 1$;
  \medskip
  
  {\em (3)}\quad $|T(f)(z)|\lesssim 1$ for any $z\in \B_{cr}(\wh)$;
  \medskip
  
  {\em (4)}\quad $|T(f)(z)|\lesssim \d(z, \wh)^{-2n}r^{2n}$ for any $z\notin \B_{cr}(\wh)$.
  \bigskip
  
  Note that conclusion {\em (a)} would then follow at once from assertions {\em (3)} and {\em (4)}; to prove conclusion {\em (b)} we first write
  $$
  \|T(f)\|^2_{L^2(\bndry D)} = \int\limits_{\B_{cr}(\wh)}\!\!\!|T(f)|^2\quad +
  \int\limits_{\bndry D\setminus \B_{cr}(\wh)}\!\! |T(f)|^2\quad =\quad  I + II,
  $$
and observe that 
$I\ \lesssim r^{2n}$ by conclusion {\em (a)} 
and similarly,
$II \lesssim r^{2n}$ by assertion {\em (4)} and 
 \eqref{E:2.16} (with $\beta = 2n$).

It remains to prove the four assertions. 
To prove the first assertion, note that since $f_0(\wh) = f(\wh)$, we may write $f(w)-f_0(w) = I + II$, where
$I = f(w) -f (\wh)$,  $II= f_0(\wh) -f_0 (w)$, and in view of our hypotheses on $f$, we only need to show that
$$
|II|= |f_0(\wh) -f_0 (w)|\lesssim 
\left(\frac{\d(w, \wh)}{r}\right)^{\!\alpha}.
$$
To this end note that 
$$
f_0(w) -f_0 (\wh)= f(\wh)\,(\tilde\chi_{r, \wh}(w)-\tilde\chi_{r, \wh}(\wh))
$$
(because $\tilde\chi_{r, \wh}(\wh)=1$). By the mean value theorem we have that the right-hand side of this identity is equal to
$$
O\left(\frac{|\mathrm{Im}\langle\dee\rho (\wh), \wh-w\rangle| + |\wh-w|^2}{c\, r^2}\right)
$$
and by \eqref{E:2.10} and \eqref{E:eg-est} this is bounded by
$$
\left(\frac{\d(w, \wh)}{r}\right)^2\leq \left(\frac{\d(w, \wh)}{r}\right)^{\!\alpha},
$$
(recall that $\alpha<1$) thus proving assertion {\em (1)}.
 To prove the second assertion, we write
$$
T(f) = T(f_0) + T(f-f_0).
$$
Now we have that $|T(f_0)(\wh)|\lesssim 1$ by \eqref{E:star-bump}, and $(f-f_0)(\wh) =0$ by the definition of $f_0$. It follows that 
$$
|T(f-f_0)(\wh)| = \left|\,\int\limits_{\bndry D}\!\!\! K(w, \wh)\, (f-f_0)(w)\, d\l (w)\right|\lesssim
\frac{1}{r^\alpha}\!\!\!\!\int\limits_{\B_r(\wh) }\!\!\!\!\d(w, \wh)^{-2n +\alpha}d\lambda (w)
$$
by our hypotheses on $T$ and $K$ along with assertion {\em (1)};
 the inequality $|T(f-f_0)|\lesssim 1$ now follows by combining the above with
\eqref{E:2.16}. To prove assertion  {\em (3)} we think of $f$ as (a multiple of) a bump function associated with a ball $\B_{c'r}(z)$ for a suitable constant $c'$, and we apply assertion {\em (2)}. We are left to prove assertion {\em (4)}. To this end, we first note that $f(z)=0$ because $z\notin\B_{cr}(\wh)$, so by our hypotheses on $T$ and $K(w, z)$ we have that
$$
|Tf(z)|\lesssim  \int\limits_{\B_r(\wh)}\!\!\!\d(w, z)^{-2n}\, d\lambda (w)
$$
(recall that $|f|\leq 1$). Now 
it follows by the triangle inequality that
$$
\d(w, z)\gtrsim \d(z, \wh)\quad\mbox{whenever}\ \ z\notin \B_{cr}(\wh),  \ \ w\in\B_r(\wh),
$$
and $c$ is sufficiently large, 
thus proving assertion {\em (4)}.
  \end{proof}
  \begin{Prop}\label{P:9}
  If $\wh\in\bndry D$ and $f$ is a normalized bump function associated with $\B_r (\wh)$, then
  \begin{equation*}
  \sup\limits_{z\in\bndry D}|\Ctrs (f) (z)|\lesssim 1
  \end{equation*}
  and
  \begin{equation*}
  \|\Ctrs (f)\|_{L^2(\bndry D, d\l)}\lesssim r^n.
  \end{equation*}
  \end{Prop}
   \begin{proof} The proof is an application of Lemma \ref{L:general}
    to $T= \Ctrs$. Then the representation for $T$, \eqref{E:T-repr-1}, follows from \eqref{E:4.6b}, while the estimate for the kernel 
   $K(w, z)$ of $T$ as above is an immediate consequence of the definition, see \eqref{E:eg-est}. It remains to check that the estimate
    $  |T(f_0)(\wh)|\lesssim 1$
  holds whenever $f_0$ is a $C^1$-smooth function supported in a ball $\B_r(\wh)$, with $|f_0|\leq 1$ and $|\nabla f_0|\leq 1/r^2$ on $\bndry D$: but this is a consequence of
   \eqref{E:4.6} and the key 
     identity \eqref{E:4.4} 
   and, in addition, 
   the fact that the kernels of the operators $\Rtr$ and $\Rtrs$, occurring in 
    \eqref{E:4.6a} and \eqref{E:4.4},  are each bounded by $c\,\d(\wh, w)^{-2n+1}$, see 
    \eqref{E:remainder-bound-a},
    while the kernel of the operator $\Etr$ is bounded by
   $c\,\d(\wh, w)^{-2n+2}$, see \eqref{E:remainder-bound-b} and \eqref{E:2.16}.
        \end{proof}
        
     \noindent{\bf Remark A.}\quad The first conclusion in Proposition \ref{P:9} also holds for $\Ctr$ 
      because the difference $(\Ctrs -\Ctr)(f)$ is easily seen to be bounded for $f$ bounded.
       \medskip
        
        The last point about $\Ctrs$ that is needed are the usual difference estimates for its kernel 
        $\g (w, z)^{-n}$. Besides the inequality $|\g (w, z)|^{-n}\lesssim \d (w, z)^{-2n}$ we have 
        \begin{equation}\label{E:4.13}
        |\g (w, z)^{-n}-g (w, z')^{-n}|\lesssim \frac{\d( z, z')}{\d (w, z)^{2n +1}}\quad 
           \end{equation}
         and
        \begin{equation}\label{E:4.13p}
        |\g (w, z)^{-n}-\g (w', z)^{-n}|\leq c_\epsilon\, \frac{\d(w, w')}{\d (w, z)^{2n +1}} 
             \end{equation}
              whenever  $\d(w, z)\geq c\,\d(z, z')$
 for an appropriate large constant $c$. In fact when $\d(w, z)\geq c\,\d(z, z')$ for large $c$, one has that $|\g(w, z)|\approx |\g (w, z')|$, and so
 $$
 |\g (w, z)^{-n}-g (w, z')^{-n}|\lesssim  \frac{|\g (w, z)-g (w, z')|}{|\g(w, z)|^{n+1}}\, .
 $$
 Thus, to prove
 \eqref{E:4.13} we need only to invoke \eqref{E:star} (applied to $A(w, z) =\ge (z, w)$, however with the roles of $w$ and $z$ interchanged with one another) and the fact that $|\g(w, z)|\approx \d(w, z)^2$.
 The inequality \eqref{E:4.13p} is proved similarly.
 
We may now set
 $$T =\Ctrs,$$ 
 and apply the $T(1)$-theorem to establish the $L^p (\bndry D, d\l)$-boundedness of such $T$.
 In the terminology of \cite[Chapter IV]{C-2} we identify the space $X$ with $\bndry D$, points $x, y$ in $X$ with points $z, w$ in $\bndry D$, the quasi-distance $\rho$ with $\d$, the measure $\mu$ with $\l$, and the kernel $K$ with $\g(w, z)^{-n}$.

 Then by what we have just shown 
 $T(1)$ and $T^*(1)$ are continuious functions on $\bd D$. 
 Also $|(Tf_1, f_2)|\lesssim r_1^n\,r_2^n$, whenever $f_1$ and $f_2$ are normalized bump functions
 associated to boundary balls $\B_{r_1}(\wh_1)$ and $\B_{r_2}(\wh_2)$, respectively.  This follows immediately from Proposition \ref{P:9}.
 In view of \eqref{E:4.6b}, the kernel
 $K(w, z) =\g(w, z)^{-n}$ has the property that
 $$
 (T f, h) = \int\limits_{\bndry D\times \bndry D}\!\!\!\!
 K(w, z)\, f(w)\, \bar h (z) \, d\l (w)\,d\l (z)
 $$
 whenever the functions $f$ and $h$ have disjoint support and are H\"older in the sense of \eqref{E:3.4}. Also, $K(w, z)$ satisfies the difference conditions
 $$
 |K(w, z) - K (w, z')|\lesssim \frac{\d (z, z')^\alpha}{\d(w, z)^{2n +\alpha}}\quad \mbox{when}\quad
 \d(w, z)\geq c\,\d(z, z')
 $$
 and
 $$
 |K(w, z) - K (w', z)|\lesssim \frac{\d (w, w')^\alpha}{\d(w, z)^{2n +\alpha}}\quad \mbox{when}\quad
 \d(w, z)\geq c\,\d(w, w').
 $$
 This is because $\g(w, z)^{-n}$ satisfies these properties for $\alpha =1$ by
 \eqref{E:4.13} and \eqref{E:4.13p}, and hence for $\alpha<1$ and in particular for 
 \begin{equation}\label{E:alphapositive}
 \alpha > 0.
 \end{equation}
  
 From these, \cite[Theorem 13]{C-2} guarantees that $T$ extends to a bounded linear 
  operator
  on $L^p(\bndry D, d\l)$, for each $1<p<\infty$. However 
  $$
  \Ctr\ =\  
   T +\Rtr
  $$
  by   \eqref{E:4.6a},
 since $T=\Ctrs$. Next recall that if $R(w, z)$ is the kernel of $\Rtr$, it follows by the description of $\Rtr$ (essentially given in \eqref{E:4.3}) that 
  $|R(w, z)|\lesssim \d (w, z)^{-2n+1}$, see \eqref{E:remainder-bound-a},
  and as a result
  $$
  \sup\limits_{z\in\bndry D}\int\limits_{w\in\bndry D}\!\!|R(w, z)|\, d\l (w)<\infty
  $$
  and
  $$
  \sup\limits_{w\in\bndry D}\int\limits_{z\in\bndry D}\!\!|R(w, z)|\, d\l (z)<\infty\, .
  $$
  Now it is well known that if $\mathcal T$ is an operator whose kernel $\mathcal T(w, z)$
  satisfies
  \begin{equation}\label{E:schur-1}
  \sup\limits_{z\in\bndry D}\!\!\int\limits_{w\in\bndry D}\!\!\!|\mathcal T(w, z)|\, d\l (w)\leq 1\, ,
\ 
  \mbox{and}\,
 \ 
  \sup\limits_{w\in\bndry D}\!\!\int\limits_{z\in\bndry D}\!\!\!|\mathcal T(w, z)|\, d\l (z)\leq 1
 \end{equation}
then 
\begin{equation}\label{E:schur-2}
\|\mathcal T\|_{L^p\to L^p}\leq 1\, .
\end{equation}
 Hence Theorem \ref{T:1} is proved if we take $\mathcal T = c\,\Rtr$.

\subsection{Further regularity}
It will be useful to have the following regularity property of $\Ctr$.
\begin{Prop}\label{P:13}
For any $0 < \alpha < 1$, the transform $\Ctr: f\mapsto \Ctr (f)$  preserves the space of H\" older-like functions satisfying condition \eqref{E:3.4}.
\end{Prop}
\begin{proof} 
The proof of this result follows the same lines as \cite[Proposition 6.3]{LS-4}
and therefore we shall be brief.
Fix $\zp_1$ and $\zp_2$ in $\bndry D$, and consider the boundary ball $\B_r(\zp_1) =\{ w\in\bndry D\ :\  \d (\zp_1, w)<r\}$ with radius $r= \d (\zp_1, \zp_2)$,
and let 
$$
\tilde\chi_{r, \zp_1}(w) =
\chi\left(\frac{\mathrm{Im}\langle\dee\rho (w), w-z_1\rangle}{c\,r^2} + i \frac{|w-z_1|^2}{c\, r^2}\right)
$$ 
be the special cutoff function, supported in this ball, that was constructed in the proof of 
Lemma \ref{L:general}
(with the center now at $\zp_1$). 
At this stage one invokes \eqref{E:Ctr-rep} for $z=z_j$, $j=1,2$, and thus one writes
$\Ctr(f)(\zp_j)$
as follows:
\begin{equation*}
\Ctr(f)(\zp_j) = I_j + II_j + f(\zp_j),\quad j=1, 2
\end{equation*}
where
\begin{equation*}
I_j=\int\limits_{w\in\bndry D}\!\!\!\!
C(w,\zp_j)\tilde\chi_{r, \zp_1}(w)(f(w)-f(\zp_j))
\end{equation*}
and
\begin{equation*}
II_j= \int\limits_{w\in\bndry D}\!\!\!\!
C(w,\zp_j)(1- \tilde\chi_{r, \zp_1}(w))(f(w)-f(\zp_j))
\end{equation*}
with $C(w, z)$  the kernel of $\Ctr$.
The first observation is then that each of $|I_1|$ and $|I_2|$ is majorized by a constant multiple of $\d(\zp_1, \zp_2)^\alpha$ (this is because the integrands are majorized by $\d(w, \zp_j)^{-2n+\alpha}$, and then one uses \eqref{E:2.16} with $\beta =\alpha$.) Next one shows that $|II_1 - II_2|$ is also majorized by 
a constant multiple of $\d(\zp_1, \zp_2)^\alpha$. To see this, one further decomposes the term $II_2$ as follows
\begin{equation*}
II_2 = \widetilde{II}_2 
 +
(f(\zp_1)-f(\zp_2))\!\!\!\!\!\int\limits_{w\in\bndry D}\!\!\!\!
C(w,\zp_2)(1- \tilde\chi_{r, \zp_1}(w))\,
\end{equation*}
with
\begin{equation*}
\widetilde{II_2} =  \int\limits_{w\in\bndry D}\!\!\!\!
C(w,\zp_2)(1- \tilde\chi_{r, \zp_1}(w))(f(w)-f(\zp_1))\, .
\end{equation*}
Then one observes that the difference $|II_1 - \widetilde{II}_2|$ is  majorized by
$\d(\zp_1, \zp_2)^\alpha$ by invoking the decomposition $\Ctr =\Ctrs +\Rtr$ 
and then using the estimate \eqref{E:4.13} for 
$|\g^{-n}(w, \zp_1) - \g^{-n}(w, \zp_2)|$ (the difference estimate for the kernel of $\Ctrs$)
 and the similar but easier estimate \eqref{E:remainder-bound-b} for the remainder operator $\Rtr$, along with the integral estimate \eqref{E:2.16} with $\beta =1$.
Finally, the integral in the remaining term in $II_2$, 
$\int
C(w,\zp_2)(1- \tilde\chi_{r, \zp_1}(w))$, is uniformly bounded, because $\Ctr (1) =1$ and $\Ctr(\tilde\chi_{r, \zp_1})$ is uniformly bounded by
Remark A,
 since $\tilde\chi_{r, \zp_1}$
 is a bump function.
 The proof of Proposition \ref{P:13}
 is concluded.
\end{proof}
\smallskip

{\bf Remark B.}\quad The proof of Proposition \ref{P:13} also shows that $\Ctrs: f\mapsto \Ctrs (f)$  preserves the space of H\" older-like functions satisfying condition \eqref{E:3.4}.
So in particular $\h =\Ctrs (1)$  is H\"older-continuous of order $\alpha$ in the sense of \eqref{E:3.4} for any $0<\alpha<1$, and one could in fact prove that the same is true for $\ha =(\Ctrs)^*(1)$; see \cite[Corollary 4 and (5.16)]{LS-4} for a similar result. 
These improved cancellation conditions would then allow to reduce the 
application of the $T(1)$-theorem for $\Ctrs$ (and therefore $\Ctr$) to the simpler situation when $T(1)=0=T^*(1)$, as was done in 
\cite[Section 6.3]{LS-4}
for the Cauchy-Leray integral of a strongly $\C$-linearly convex domain.
However in the present context we will not pursue this approach, as a further application of the $T(1)$-theorem (for a related operator) will be needed when we deal with the Cauchy-Szeg\H o projection in Part II below, but in that context the cancellation conditions for $T(1)$ and $T^*(1)$ cannot be improved beyond continuity
on $\bndry D$.

\section*{Part II: The Cauchy-Szeg\H o projection}\label{PIII}

We now come to
the main result of this paper:
 the $L^p (\bndry D)$-regularity of the Cauchy-Szeg\H o projection for
 $1<p<\infty$ (Theorem \ref{T:6.1.1}).

As mentioned earlier, in defining the Cauchy-Szeg\H o projection it is imperative to specify the underlying measure for $\bndry D$ that arises in the notion of orthogonality that is used. To put the matter precisely, suppose $d\sigma$ is the induced Lebesgue measure on $\bndry D$, and consider also a weight function $\omega$
which is strictly positive and continuous on $\bndry D$. Then the Lebesgue spaces, 
$L^p(\bndry D, d\sigma)$ and $L^p(\bndry D, \omega\,d\sigma)$
(with  norms 
$\|f\|_{L^p(\bndry D,\, d\sigma)}=(\int|f|^pd\sigma)^{1/p}$ and $\|f\|_{L^p(\bndry D,\, \omega\,d\sigma)}=(\int|f|^p\omega\,d\sigma)^{1/p}$ respectively) 
are {\em{equivalent}} in the sense that {\em 1.} the two spaces consist of the same elements, and {\em 2.}
 the two norms are comparable, that is
 $$
 \|f\|_{L^p(\bndry D,\, d\sigma)} \approx \|f\|_{L^p(\bndry D,\, \omega\,d\sigma)}.
 $$

So in this way we can speak about the boundedness on $L^p(\bndry D)$
of the 
Cauchy transform $\Ctr$ we studied in Part I,
 without here specifying the particular strictly positive continuous weight function used.
   However the measures $d\sigma$ and $\omega\,d\sigma$ give rise to different inner products on $L^2 (\bndry D)$, that are respectively 
$$
(f_1, f_2)_1=\int\limits_{\bndry D}\!\!f_1\bar{f_2}\,d\sigma\quad\mbox{and}\quad
(f_1, f_2)_\omega=\int\limits_{\bndry D}\!\!f_1\bar{f_2}\,\omega\,d\sigma\, .
$$
With this in mind, we define the Cauchy-Szeg\H o projection $\Sw$ to be the orthogonal projection 
of $L^2(\bndry D, \omega\,d\sigma)$, where orthogonality is taken with respect to the inner product $(\cdot, \cdot)_\omega$, onto the {\em Hardy Space} $\htwom$, which we presently
define as the closure in $L^2(\bndry D, \omega\,d\sigma)$ of the set of boundary values
   of those functions that are continuous on $\bar{D}$ and holomorphic 
 in $D$.
  (Further characterizations and representations of $\mathcal H^2(\bndry D, \omega\,d\sigma)$ and more generally, of $\hp (\bndry D, \omega\,d\sigma)$, are given in \cite{LS-5}.) In view of the above, the usual Cauchy-Szeg\H o projection is $\Sone$, and the Cauchy-Szeg\H o projection with respect to the Leray-Levi measure $d\lambda$ is then $\Sl$, where
$d\lambda = \Lambda\,d\sigma$, see \eqref{E:2.15}. While there is no simple and direct link 
connecting any two of these projections, the Cauchy-Szeg\H o projection $\Sl$ with respect to the Leray-Levi measure provides the way to understanding the projection $\Sone$ (and all the others, the totality of $\mathcal S_\omega$). For this reason we study $\Sl$ first.

\section{The Cauchy-Szeg\H o projection: case of the Leray-Levi measure}\label{S:5}

\subsection{Statement of the main results} In this section it will be convenient 
 to simplify the notation, dropping the subscript $\lambda$, and to denote $\Sl$ by $\S$, $L^p(\bndry D, d\lambda)$ by $L^p(\bndry D)$, and so forth.
Our main results are as follows. 
\begin{Thm}[Main Theorem]\label{T:6.1.1}
The operator $\S$, initially defined on $L^2(\bndry D)$, extends to a bounded operator on $L^p(\bndry D)$, for $1<p<\infty$.
\end{Thm}
In the next section we prove the general result
\begin{Thm}\label{T:6.1.2}
The same conclusion holds for $\Sw$, whenever $\omega$ is a strictly positive and continuous function on $\bndry D$.
\end{Thm}
\subsection{Outline of the proof of Theorem \ref{T:6.1.1}}\label{SS:mainthm}
 In what follows it will be important to keep track of the dependence of the Cauchy transform on $\epsilon$, so we will revert to writing the Cauchy transform as $\Ctre$, with $\Ctre$ equaling the $\Ctr$ that appears in Section \ref{S:4}.
   Similarly we will write $\Rtrse$ for $\Rtrs$, $\ge$ for $g$, etc.
 
  As we pointed out in the introduction, one of the main thrusts in the proof of Theorem \ref{T:6.1.1} is a comparison of
  the Cauchy-Szeg\H o projection 
   with the Cauchy transforms $\{\Ctre\}$ that were discussed in Part I. Such comparison is effected in Proposition \ref{P:6.2.1} below,
  whose proof
  is given in \cite{LS-5}.
\begin{Prop}\label{P:6.2.1}
As operators on $L^2(\bndry D)$ we have
\medskip

(a)\quad $\Ctre\S =\S$
\medskip

(b)\quad $\S\Ctre = \Ctre$
\medskip
\end{Prop}

\noindent This proposition immediately implies
\medskip

$
(c)\quad \S(I+ \Ctre -\Ctre^*)=\Ctre\, ,
$
\medskip

as can be seen by taking adjoints
  of identity {\em (a)} which gives
 $\S\Ctre^*=\S$, and subtracting this from {\em (b)}.
Here the
super-script $^*$ 
denotes the adjoint with respect to the inner product
$$
(f_1, f_2)_\lambda=\int\limits_{\bndry D}\!\!f_1\bar{f_2}\,d\lambda\, ,
$$
for which $\S^*=\S$.
An identity analogous to {\em (c)} was used in \cite{KS-2} to study the Cauchy-Szeg\H o projection in the  
situation when $D$ is smooth (and strongly pseudo-convex). In that 
situation there was only the case $\epsilon =0$ and 
 the analogue of 
$\Ctre -\Ctre^*$ was ``small'' (in fact smoothing),
 and $I+\Ctre-\Ctre^*$ 
was  ``inverted'' by a partial
Neumann series.  By contrast, when $D$ is of class $C^2$ (as opposed to smooth) we will see below that $\|\Ctre\|_{L^p\to L^p}$ may tend to $\infty$ as $\epsilon\to 0$ and, similarly, that no appropriate control on the size of $\Ctre -\Ctre^*$  can be expected as $\epsilon \to 0$. What works instead is to truncate the transform $\Ctre$. We shall write 
\begin{equation*}
\Ctre = \Ctrei +\Etrei\,
\end{equation*}
where
the kernel of $\Ctrei$ agrees with that of $\Ctre$ when 
$\d(w, z)\lesssim \iota$
and vanishes when
$\d(w, z)\gtrsim\iota$.
If this cutoff is done appropriately, we then have the following key fact.
\begin{Prop}\label{P:6.2.2}
For any $\epsilon>0$ there is an 
$\iota (\epsilon)>0$,
so that when 
$\iota\leq \iota (\epsilon)$
\begin{equation}\label{E:6.2.0a}
\|\Ctrei-(\Ctrei)^*\|_{L^p\to L^p}\lesssim \epsilon^{1/2}M_p,\quad 1<p<\infty\, .
\end{equation}
Here the bound $M_p$ is independent of $\epsilon$ and depends on $p$ as follows:
$$M_p = \frac{p}{p-1} +p.$$
\end{Prop}
{\bf Remark C.} The proof also shows that $\epsilon^{1/2}$ can be replaced by 
$\epsilon^{\beta}$ for any $0<\beta<1$ (but not $\beta =0$, nor $\beta=1$). However any bound that tends to zero with $\epsilon$ will suffice in 
our application below.
\medskip

\noindent {\em Proof of Theorem \ref{T:6.1.1}.}\quad Let us see how Proposition \ref{P:6.2.2} proves 
Theorem \ref{T:6.1.1}. Consider the case $1<p\leq 2$. Since $\Ctre = \Ctrei + \Etrei$ we have by Proposition \ref{P:6.2.1} (in fact identity {\em (c)}) that
\begin{equation}\label{E:6.2.1}
\Ctre +\S(\Etrei)^* - \S\Etrei = \S\big(I+\Ctrei-(\Ctrei)^*\big).
\end{equation}
However $I+\Ctrei-(\Ctrei)^*$ is invertible as a bounded operator on $L^p (\bndry D)$ when 
$\epsilon$ and $\iota$ are taken sufficiently small, by Proposition \ref{P:6.2.2}, as can be seen
by applying a Neumann series when $\epsilon^{1/2}M_p<<1$.

Next, the kernel of $\Etrei$ is supported where
$\d (w, z)\gtrsim\iota$, 
and can be seen to be bounded by 
$ c_\epsilon's^{-1}\d(w, z)^{-2n+1}$, see \eqref{E:remainder-bound-a}.
Thus $\Etrei$ maps $L^1(\bndry D)$ to $L^\infty (\bndry D)$ (but the norm $\|\Etrei\|_{L^1\to L^\infty}$ is not bounded as $\epsilon$ and 
$\iota$ 
tend to zero!). The same may be said of
$(\Etrei)^*$.

So $\Etrei$ and $(\Etrei)^*$ map $L^p(\bndry D)$ to $L^2(\bndry D)$, while $\S$ maps
 $L^2(\bndry D)$ to $L^2(\bndry D)$, and hence it maps $L^2(\bndry D)$ to $L^p(\bndry D)$, because
 $p\leq 2$. (Here we use the hypotesis that $D$, and hence $\bndry D$, is bounded.)
 
 Altogether then the left-hand side of \eqref{E:6.2.1} is bounded on $L^p(\bndry D)$, in view of the corresponding boundedness of $\Ctre$, Theorem \ref{T:1}. Applying $(I+(\Ctrei -(\Ctrei)^*)^{-1}$
 to both sides of \eqref{E:6.2.1} yields the boundedness of $\S$ on $L^p(\bndry D)$, when $1<p\leq 2$. The case $p>2$ follows by duality.\qed
 
 At this point we digress briefly to state, for the sake of comparison, a companion result that will be needed in Section \ref{S:last}, when we will deal with the Cauchy-Szeg\H o projection for $L^2(\bndry D, d\sigma)$.
 
 \begin{Prop}\label{P:X}
For any $\eps>0$ there is an $\iota (\eps)>0$, so that when $\iota\leq \iota (\eps)$
\begin{equation}\label{Eq:star}
\|\,\Ctrei\,\|_{_{L^p\to L^p}}\lesssim M_p\, ,\quad \mbox{for}\ 1<p<\infty.
\end{equation}
\end{Prop}
Note that as opposed to \eqref{E:6.2.0a} in Proposition \ref{P:6.2.2}, here we cannot have a gain in $\eps$ by making $\iota$ small.

\subsection{Proof of Proposition \ref{P:6.2.2}}
\label{SS:6.3}
As stated above,
 the norm of the Cauchy transform $\Ctre$ may be unbounded when $\epsilon\to 0$. This is due to the fact that we have replaced the continuous functions $\dee^2\rho(w)/\dee w_j\dee w_k$ by their $C^1$ approximations $\tau^\epsilon_{j, k}(w)$, and the quantities
\begin{equation}\label{E:def-ceps}
c_\epsilon =\sup\limits_{\stackrel{w\in\bndry D}{1\leq j, k\leq n}}
|\nabla\tau^\epsilon_{j, k}(w)|
\end{equation}
which first occurred in \eqref{E:starp} and appear again below, will in general tend to infinity as $\epsilon\to 0$ in a manner that reflects 
the modulus of continuity of the $\dee^2\rho(w)/\dee w_j\dee w_k$. The constants $c_\epsilon$ first appeared in \eqref{E:starp};
  they also occurred implicitly at various other stages when first derivatives of the $\tau^\epsilon_{j, k}(w)$ were involved. Among these instances are the definition of the Cauchy integral \eqref{E:3.2} via the form $\deebar G$ (see \eqref{E:3.1}), and the coefficients 
$\alpha_j(w, z)$ that enter in the formula \eqref{E:4.3}. These facts are all working against us as we try to control the growth of $\Ctre$ (or $\Ctrei$). However, not hindering us is the simple observation
that was used before, see \eqref{E:2.5}, namely the fact that
$$
|\ge (w, z)|\approx |\g_0 (w, z)| = \d (w, z)^2\, ,
$$
with the implied bounds not dependent on $\epsilon$. 

Moreover, what really helps us are two lemmas below. The first shows that under the right 
circumstances we can remove the bound $c_\epsilon$ from \eqref{E:starp}.
\begin{Lem}\label{L:6.3.1}
For every $\epsilon>0$ there is  an
$\iota =\iota (\epsilon)$,
so that if 
$\iota\leq\iota(\epsilon)$ 
and
  $\d(w, z)\leq \iota$, $\d(w', z)\leq \iota$, 
 then
 \begin{equation}\label{E:6.3.2}
 |\ge(w, z) - \ge (w', z)|\lesssim \d(w, w')^2 + \d(w, w')\,\d(w, z)\, .
 \end{equation}
\end{Lem}
Indeed,
recalling
 the proof of \eqref{E:starp} we see that
   it gives
$$
|\ge(w, z) - \ge (w', z)|\lesssim \d(w, w')^2 + \d(w, w')\,\d(w, z) + |\,II\,|\, ,
$$
with $II=Q_w(w'-z)-Q_{w'}(w'-z)$, and in \eqref{E:star} it was proved  that 
$$
|\,II\,|\,\lesssim c_\epsilon\, |w-w'|\,|w'-z|^2
$$
where $c_\epsilon$ is as in \eqref{E:def-ceps}. Also $|w-w'|\lesssim \d(w, w')$, $|w'-z|\lesssim 
\d(w', z)$, and $\d(w', z)\lesssim \d(w, w') + \d(w, z)$ by the triangle inequality.

Thus if we take 
$\iota_0(\epsilon) = c_\epsilon^{-1}$, 
then
$$
|\,II\,|\lesssim c_\epsilon\d(w, w')\,\d(w', z)^2\leq \d(w, w')\,\d(w', z)\lesssim
\d(w,w')^2+\d(w, w')\,\d(w, z),
$$
which establishes \eqref{E:6.3.2}.

The near symmetry of the Cauchy kernel (up to negligeable errors) was essential in the treatment
 of the Cauchy-Szeg\H o projection in the case of smooth domains, as in \cite{KS-2}. Only a vestige of this
 fact remains in our case and it is given as follows.
 We recall that $|\ge(w, z)|\approx|\ge(z, w)|\approx\d(w, z)^2$.
 \begin{Lem}\label{L:6.3.2}
 For every $\epsilon>0$ there is an 
  $\iota(\epsilon)>0$,
 so that if 
  $\iota\leq \iota(\epsilon)$
 then
 \begin{equation}\label{E:6.3.3}
 |\ge(w, z)-\bar{\ge}(z, w)|\lesssim \epsilon\,\d(w, z)^2
 \end{equation} 
 whenever 
  $\d(w, z)\lesssim \iota$.
 \end{Lem}
 The proof of \eqref{E:6.3.3} uses the modulus of continuity of the second derivatives of $\rho$ and
  is essentially given in \cite[Section 2.3]{LS-2}. It is shown there that the left-hand side of
   \eqref{E:6.3.3} is majorized by $c_0\epsilon\,|w-z|^2$ if $|w-z|\leq \delta(\epsilon)$.
    Since $|w-z|\leq c_1\d(w, z)$, we obtain \eqref{E:6.3.3} when we take
         $\iota(\epsilon)=c'\delta(\epsilon)$ 
     for an appropriate constant $c'$.
 \bigskip
 
   {\em Proof of Proposition \ref{P:6.2.2}}.\quad
 We begin by describing the truncations of the operators that we will deal with. If $\Ctre$ is our Cauchy transform, then for $\iota>0$ 
 its truncated version $\Ctrei$ is defined by
 \begin{equation}\label{E:6.4.1a}
 \Ctrei (f)(z) = \Ctre \big(f(\cdot)\, \chi_\iota(\cdot, z)\big)(z),\quad z\in\bndry D.
 \end{equation}
 Here $\chi_\iota (w, z)$ is the symmetrized 
version of the cut-off function
 given
 in the proof of Lemma \ref{L:general}.
 It is defined by 
$$
\chi_\iota (w, z) = \tilde\chi_{\iota, w}( z)\, \tilde\chi_{\iota, z}(w), 
$$
where we recall that
$$
\tilde\chi_{\iota, w}(z) = 
\chi\left(\frac{\mathrm{Im}\langle\dee\rho (w), w-z\rangle}{c\,\iota^2}\,+i\, \frac{|w-z|^2}{c\, \iota^2}\right)
$$
with $\chi (u+iv)$ as in the proof of Lemma \ref{L:general}.

 Instead of dealing directly with $\Ctre$ and its truncation $\Ctrei$, we work with $\Ctrse$ and its corresponding truncation $\Ctrsei$.
 We do so  because identities \eqref{E:4.7a} and \eqref{E:4.7} make the choice of the formal adjoint of $\Ctrse$ obvious, as opposed to the situation for
  $\Ctre$. 
  Now, by what has been said above, the truncation $\Ctrsei$ is defined by
   \begin{equation*}
\Ctrsei (f)(z) = \Ctrse \big(f(\cdot)\, \chi_\iota(\cdot, z)\big)(z),\quad z\in\bndry D\, .
\end{equation*} 

We observe
 that the 
kernel of $\Ctrsei$ is $\ge(w, z)^{-n}\chi_\iota (w, z)$, in the sense that 
\begin{equation*}
\Ctrsei (f)(z) =
\int\limits_{\bndry D}\!\! \ge(w, z)^{-n}\chi_\iota(w, z)f(w)\, d\lambda(w)\, ,
\end{equation*}
whenever $f$ satisfies the H\"older regularity 
\eqref{E:3.4} and $z\in\bndry D$ is such that
$f(z)=0$;
  this is because a 
  corresponding
  formula holds for $\Ctrse$, see \eqref{E:4.6b}.
Using the reasoning of Proposition \ref{P:7} we can see that there 
is an adjoint $(\Ctrsei)^*$, defined on functions that satisfy the H\"older regularity 
\eqref{E:3.4}, such that
$$
\big((\Ctrsei)^* (f_1), f_2\big) = \big(f_1, (\Ctrsei)^* f_2\big)
$$
whenever $f_1$ and $f_2$ are a pair of such functions. Moreover
\begin{equation*}
(\Ctrsei)^* (f)(z) = (\Ctrse)^* \big(f(\cdot)\,\chi_\iota(\cdot, z)\big)(z)\, ,
\end{equation*}
and
\begin{equation*}
(\Ctrsei)^* (f)(z) = \int\limits_{\bndry D}\!\bar{\ge}(z, w)^{-n}\chi_\iota(w, z)f(w)\,d\lambda (w)\,
\end{equation*}
with the second identity holding whenever $f(z)=0$.

With these in place we turn to the proof of the analogue of \eqref{E:6.2.0a}, where the operator $\Ctrei$ is replaced by $\Ctrsei$. To simplify the notation we set 
$$
\Dsei = \Ctrsei - (\Ctrsei)^*.
$$
We assert that the operator $\Dsei$, initially defined on H\"older-type functions, extends to a bounded operator on $L^p(\bndry D)$, with 
\begin{equation}\label{E:6.4.4}
\|\Dsei\|_{L^p\to L^p}\lesssim \epsilon^{1/2}M_p
\end{equation}
for any $1<p<\infty$, as long as $\iota\leq\iota(\epsilon)$. (The relevance of the exponent $1/2$ is discussed in Section \ref{SS:6.5} below, see Remark D there.)

	Before coming to the proof of \eqref{E:6.4.4} let us see how this implies Proposition \ref{P:6.2.2}.  
	To this end, we begin by
		writing
	\begin{equation*}
	\Ctrei = \Ctrsei + \Rtrsei\, 
	,\quad\mbox{and}\quad 
	(\Ctrei)^* = (\Ctrsei)^* + (\Rtrsei)^*\, .
		\end{equation*}
	From these
	 it follows that
	 $$
	 \Ctrei - (\Ctrei)^* =\Dsei + \Atrei
	 $$
	 with
	 $$
	 \Atrei = \Rtrsei  - (\Rtrsei)^*.
	 $$
	Now each of the components of the kernel of $\Atrei$ is bounded by
	$$
		 \tilde c_\epsilon\d(w, z)^{-2n+1}\chi_\iota(w, z),
		$$
	see \eqref{E:remainder-bound-a}.
	From this and \eqref{E:2.16} it follows that the operator $\mathcal T=s^{-1}\Atrei$ satisfies the estimates \eqref{E:schur-1} and these in turn imply that
  $$
  \|\Atrei\|_{L^p\to L^p}\leq \tilde c_\epsilon O(\iota)\, ,
  $$
  see \eqref{E:schur-2}. While the constant $\tilde c_\epsilon$ may be large as $\epsilon\to 0$, we have that and $\tilde c_\epsilon O(\iota) = O(\epsilon^{1/2})$, if $\iota$ is sufficiently small. This proves 
  Proposition \ref{P:6.2.2} and as the argument in Section \ref{SS:mainthm} also shows, gives the proof of Theorem \ref{T:6.1.1}, assuming the validity of \eqref{E:6.4.4}.

 \subsection{Kernel estimates and cancellation conditions of $\Dsei$}\label{SS:6.5}
We will apply the $T(1)$ theorem to
prove \eqref{E:6.4.4}, and so we need to come to grips with two issues. First, there are the difference inequalities for the kernel of $\Dsei$. Here, as opposed to the situation for
$\Ctrs=\Ctrse$ and $(\Ctrs)^*= (\Ctrse)^*$ given in \eqref{E:4.13} and \eqref{E:4.13p},
 we need to see an appropriate gain in $\epsilon$ (of order $\epsilon^{1/2}$). Second, for the cancellation properties, we will see a gain, in fact of order $\epsilon$, beyond what holds for $\Ctrs$ as stated in Proposition
 \ref{P:9}. 
 
 Let us now deal with the kernel of $\Dsei$, which we denote by $\Ksei (w,z)$. According to our previous discussion it is 
 \begin{equation*}
 \Ksei (w,z) = (\ge (w, z)^{-n} -\bar{\ge}(z, w)^{-n})\,\chi_\iota (w, z)\, .
 \end{equation*}
 Our objective is to prove the inequality:
 \begin{equation}\label{E:6.4.7}
  |\Ksei (w,z) - \Ksei (w',z)|\lesssim 
  \epsilon^{1/2}\frac{\d (w, w')^{1/2}}{\d (w, z)^{2n + 1/2}}\quad \mbox{for any} \ 
  \iota\leq\iota(\epsilon),
  \end{equation}
  under the proviso that 
  $$
  \d(w, z)\geq c\,\d(w,w')
   $$
  for a sufficiently large constant $c$.
  We claim first that 
\begin{equation}\label{E:6.4.5}
|\Ksei (w,z)|\lesssim\epsilon\,\d(w, z)^{-2n},\quad \mbox{if}\ \iota\leq\iota(\epsilon).
\end{equation}
In fact since $|\ge (w, z)|\approx |\ge(z, w)| \approx \d(w, z)^2$, we have
$$
|\ge (w, z)^{-n} -\bar{\ge}(z, w)^{-n}|\lesssim \frac{|\ge (w, z) -\bar{\ge}(z, w)|}{|\ge(z, w)|^{n+1}}\, .
$$
However the support of $\chi_s(w, z)$ restricts matters to $w$ in $\B_\iota(z)$, and thus 
inequality \eqref{E:6.3.3}
ensures that 
$$
|\Ksei (w,z)|\lesssim\epsilon\,\frac{\d(w, z)^2}{\d(w, z)^{2n+2}}
=
\epsilon\,\d(w, z)^{-2n},\quad 
\mbox{when}\ \d(w, z)\leq\iota,
$$
and \eqref{E:6.4.5} is thus established. Note that under our assumptions $\d(w, z)\approx \d(w', z)$ and hence we also have
\begin{equation}\label{E:6.4.5.a}
|\Ksei (w',z)|\lesssim\epsilon\,\d(w, z)^{-2n},\quad \mbox{if}\ \ \iota\leq\iota(\epsilon).
\end{equation}
In particular this and \eqref{E:6.4.5} give
\begin{equation}\label{E:6.4.5.a}
|\Ksei (w,z) - \Ksei (w',z)|\lesssim \epsilon\,\d(w, z)^{-2n}\, \quad \mbox{if}\ \ \iota\leq\iota(\epsilon).
\end{equation}
We next claim that we also have
\begin{equation}\label{E:6.4.6}
|\Ksei (w,z) - \Ksei (w',z)|\lesssim \frac{\d (w, w')}{\d (w, z)^{2n + 1}}\quad \mbox{as long as}\ \ \iota\leq \iota (\epsilon).
\end{equation}
This again follows by treating the constituents of $\Ksei (w,z) - \Ksei (w',z)$ separately. Specifically, we first rearrange the terms of $\Ksei (w,z) - \Ksei (w',z)$ into three separate groups, namely {\em 1.}, a group of terms involving the difference $\ge(w, z) - \ge(w', z)$; {\em 2.}, terms 
that have  $\bar{\ge}(z, w) - \bar{\ge}(z, w')$; and finally {\em 3.}, terms with $\chi_\iota (w, z) -\chi_\iota (w', z)$. We then provide estimates for each such group.
  
   To see that inequality \eqref{E:6.4.6} holds for the first group, 
    we argue as in the proof of \eqref{E:4.13} and \eqref{E:4.13p},
    except that now we invoke Lemma
    \ref{L:6.3.1} which
    avoids
       the undesirable factor $c_\epsilon$. The argument for the terms involving $\bar{\ge}(z, w) - \bar{\ge}(z, w')$ is similar, but here we do not need Lemma \ref{L:6.3.1}. Next, to treat the terms involving 
    $\chi_\iota (w, z) -\chi_\iota (w', z)$ we note that this difference vanishes, unless the equivalent quantities
    $$
    |\ge(w, z)|\approx |\ge(z, w)|\approx \d(w, z)^2
    $$
    are comparable to $\iota^2$, and in such case this difference is easily seen to be
     dominated by a multiple of of
     $$
     \frac{1}{\iota^2}\big(\d(w, w')^2+\d(w, w')\,\d(w', z)\big);
     $$
     since $\iota\approx\d(w, z)\approx \d(w', z)$, here we have that 
     $$|\chi_\iota (w, z) -\chi_\iota (w', z)|\lesssim \frac{\d(w, w')}{\d(w, z)}.$$
     Combining the above we obtain \eqref{E:6.4.6}.
  Finally we combine \eqref{E:6.4.5.a} with \eqref{E:6.4.6} to get 
  the desired regularity \eqref{E:6.4.7} 
    via the geometric mean of these two inequalities.
       Furthermore, since 
     $$\Ksei (w,z) = -\,\bar{\Ksei} (z, w)\, ,$$
  the result \eqref{E:6.4.6} and therefore \eqref{E:6.4.7} with $w$ and $z$ interchanged also follows.
This concludes the proof of the difference inequalities for $\Dsei =\Ctrsei - (\Ctrsei)^*$.
  
  The proof of the cancellation properties of $\Dsei$ is based on the following lemma.
  \begin{Lem}\label{L:6.5.1}
  Suppose $f\in C^1(\bndry D)$, with $|f(w)|\leq 1$ for all $w\in \bndry D$. Then for 
  every $\epsilon>0$ there is an $\iota (\epsilon)>0$, so that if 
  $\iota\leq\iota (\epsilon)$ and $z\in \bndry D$ we have
  $$
  |\Dsei (f) (z)|\,\lesssim \
  \epsilon + \epsilon\!\!\int\limits_{\bndry D}\!|\nabla f (w)|\,\d(w, z)^{-2n+2}d\lambda (w)\, .
  $$
 The implied constant is independent of $\epsilon$, $\iota$ (and $f$ and $z$).
   \end{Lem}
  \begin{proof} We begin by  considering the kernel
  $$
\tilde\Ksei (w,z) = \frac{1}{1-n}(\ge (w, z)^{-n+1} -\bar{\ge}(z, w)^{-n+1})\,\chi_\iota (w, z)\, 
$$
for $w\in\bndry D$ and $z$ fixed in $D$.
 Our goal is to apply Lemma \ref{L:9} to
$$
F(w) = f(w) \tilde\Ksei (w,z)
$$
(and $\rho$ the defining function of our domain). We claim that computing 
$dF\wedge (\dee\deebar \rho )^{n-1}$ for such an $F$ will give rise to several terms. One of these will be $\Ksei (w,z) f(w) d\l (w)$ (whose integral over $\bndry D$ is precisely $\Dsei (f) (z)$), while the integrals of the other terms will provide
the required bound for $|\Dsei (f) (z)|$.

To see this, we
recall \eqref{E:4.1.a}, which gives us
\begin{equation*}
d_w\ge (w, z)=\dee\rho (w) + O_\epsilon (|w-z|)\, 
\end{equation*}
and similarly,
\begin{equation*}
d_w\bar{\ge}(z, w) = -\deebar\rho (w) + O_\epsilon (|w-z|).
\end{equation*}
Here $O_\epsilon$ indicates that the bound in the term $O_\epsilon (|w-z|)$ depends on $\epsilon$. Indeed such bound is $C\cdot (1+c_\epsilon)$ with $c_\epsilon$ as in
\eqref{E:def-ceps}. Recalling that $j^*\dee\rho=-j^*\deebar\rho$, we get
\begin{equation*}
d_w \tilde\Ksei (w,z) = \Ksei (w, z) \dee\rho (w) + O_\epsilon\big(\d(w, z)^{-2n+1}\big)\chi_\iota(w, z) +
\frac{\epsilon}{\iota^2} O\big(\d(w, z)^{-2n+2}\big)
\end{equation*}
because $|w-z|\lesssim\d(w, z)$ and $|\ge(w, z)|\approx |\ge(z, w)|\approx\d(w, z)^2$, where we recall that
$$
 \ge (w, z) =\langle\dee\r (w), w-z\rangle -\frac12\sum\limits_{j, k} \tau^\epsilon_{jk}(w)\, (w_j-z_j) (w_k-z_k)\, 
 $$
for $w$ near $z$, see \eqref{E:2.3}.
In applying Lemma \ref{L:9} with $F$ as above we keep in mind that 
when $d_w$ acts on $\chi_\iota (w, z)$ we obtain a term which is
$$
O\left(|\tilde\Ksei (w, z)\,f(w)|\,\frac{1}{\iota^2}\, \mathds{1}_\iota (w, z)\right)\, ,
$$
where $\mathds{1}_\iota (w, z)$ is the characteristic function of $\B_\iota(z)$. As a result, we get that
$$
\Dsei
(f) (z) \, = \, I + II + III\,,\quad z\in \bndry D\, ,
$$
where
$$
I= -\frac{1}{(2\pi i)^n}\!\!\int\limits_{\bndry D}\!
\tilde\Ksei (w, z)\, df(w)\wedge j^*(\deebar\dee\rho (w)^{n-1})\, ,
$$

$$ 
II= 
\int\limits_{\B_\iota(z)}\!\!\!O_\epsilon\big(\d(w, z)^{-2n+1}|f(w)|\big)
$$
and
$$
III=
 \int\limits_{\B_\iota(z)}\!\!\!O\!\left( \tilde\Ksei (w, z)\frac{|f(w)|}{\iota^2}\right)\, .
$$
Now the argument given in \eqref{E:6.4.5} shows that 
$$
|\tilde\Ksei (w, z)|\lesssim \epsilon\,\d(w, z)^{-2n+2}
$$
as long as $\d(w, z)\leq \iota(\epsilon)$. Hence we have that 
$$
|I|\lesssim \epsilon\int\limits_{\bndry D}\!\!|\nabla f(w)|\,\d(w, z)^{-2n +2} d\lambda (w).
$$
The term $II$ is bounded by 
$$
\int\limits_{\bndry D}\!O_\epsilon\big(\d(w, z)^{-2n +1}\big)\, d\lambda (w)
$$
and the term $III$ is bounded by 
$$
\frac{\epsilon}{\iota^2}\!\!\!
\int\limits_{\B_\iota (z)}\!\!\! O_\epsilon\big(\d(w, z)^{-2n +2}\big)\, d\lambda (w)\, .
$$
These two terms are bounded respectively by $\iota\cdot O_\epsilon$ and $\epsilon\cdot O(1)$ in view of \eqref{E:2.16}. This gives a bound which is a multiple of $\epsilon +\iota O_\epsilon\lesssim \epsilon$, if we take $\iota\leq \iota (\epsilon)$ with $\iota (\epsilon) =1/O_\epsilon$. The lemma is therefore proved.
  \end{proof}
  We now apply Lemma \ref{L:6.5.1} to prove the cancellation properties of the operator 
  $$
  \Dsei = \Ctrsei - (\Ctrsei)^*\, .
  $$
  We first have that the function $\Dsei (1)$ is in $L^\infty(\bndry D)$ (in fact it can be seen that it is a continuous function) and moreover
  \begin{equation}\label{E:6.5.7}
  |\Dsei (1)|\lesssim \epsilon\, ,\quad \mbox{whenever}\ \ \iota\leq \iota (\epsilon)\, ,
  \end{equation}
    with the implicit constant independent of $\epsilon$, which follows immediately from Lemma \ref{L:6.5.1}.
  
  Since by its definition $(\Dsei)^* = - \Dsei$, we also have
  \begin{equation}\label{E:6.5.8}
  |(\Dsei)^* (1)|\lesssim \epsilon\, ,\quad \mbox{whenever}\ \ \iota\leq \iota (\epsilon)\, .
  \end{equation}
  Next we argue as in Section \ref{SS: C-tr-bdd}. If $f_0$ is a $C^1$-smooth function supported in a ball $\B_r(\wh)$
    that satisfies $|f_0(w)|\leq 1$ and $|\nabla f_0(w)|\leq 1/r^2$, then we can show that
  \begin{equation}\label{E:6.5.9}
  |\Dsei (f_0)(\wh)|\lesssim \epsilon\, \quad \wh\in\bndry D.
  \end{equation}
  Indeed by  Lemma \ref{L:6.5.1}, it suffices to see that 
  $$
  \int
    |\nabla f_0 (w)|\,\d(w, \wh)^{-2n+2}\, d\lambda (w) \ \lesssim 1\, .
  $$
  But in fact our hypotheses on $f_0$ and \eqref{E:2.16} grant
  $$
 \int\!
 |\nabla f_0 (w)|\,\d(w, \wh)^{-2n+2}\, d\lambda (w)\leq 
 \frac{1}{r^2}\!\!\int\limits_{\B_r(\wh)}\!\!\d(w, \wh)^{-2n+2}\, d\lambda (w) \lesssim 1\, ,
  $$
   proving the assertion \eqref{E:6.5.9}.
  
  Altogether, the above shows that $T= \epsilon^{-1}\Dsei$ satisfies the hypotheses of Lemma \ref{L:general},
    and so its conclusion grants that
      \begin{equation}\label{E:6.5.10}
  |\Dsei (f)(z)|\lesssim\epsilon \ \mbox{for all}\ z\in\bndry D,\  \mbox{and}\quad
  \|\Dsei (f)\|_{L^2(\bndry D)}\lesssim\epsilon r^{2n}
  \end{equation}
for $ \iota\leq \iota (\epsilon)$ and  for any normalized bump function $f$ associated to a ball $\B_r(\wh)$, whenever $\wh\in\bndry D$.
  
  We recall the $T(1)$-theorem  as was used  in Section \ref{SS: C-tr-bdd}, after \eqref{E:4.13p}. 
  We now invoke the following a uniform version of this theorem.
  We consider the operators $\epsilon^{-1/2}\Dsei$. They satisfy kernel estimates and cancellation conditions 
that are uniform in $\epsilon$, 
as can be seen by \eqref{E:6.4.7}, \eqref{E:6.5.7}, \eqref{E:6.5.8} and  \eqref{E:6.5.10}.

  Under these uniform conditions one  has that the operators $\epsilon^{-1/2}\Dsei$ satisfy the bounds
  $$
  \|\epsilon^{-1/2}\Dsei\|_{L^p\to L^p}\lesssim M_p\, 
  $$
  with $ M_p$ independent of $\epsilon$.
  This proves \eqref{E:6.4.4}, thus Proposition \ref{P:6.2.2}, and therefore concludes the proof 
  of the main result, Theorem \ref{T:6.1.1}.
  
  \medskip
  
  \noindent{\bf Remark D.}   The argument that proves \eqref{E:6.4.7} also gives the inequality with the right-hand side of \eqref{E:6.4.7}
   replaced by 
  $$
  \epsilon^{\beta}\frac{\d (w, w')^{1-\beta}}{\d (w, z)^{2n + 1-\beta}}\quad \mbox{for}\ \ 0\leq \beta\leq 1\, .
  $$
Thus the power $\beta =1/2$ grants the conclusions \eqref{E:6.4.7} and \eqref{E:6.4.4},  but corresponding conclusions would also hold with $\epsilon^{1/2}$ replaced by $\epsilon^\beta$ for any $0<\beta<1$; on the other hand the cases $\beta =0$ and $\beta =1$ cannot be used, the latter because of the requirement \eqref{E:alphapositive}  that $\alpha$ (that is, $1-\beta$) be positive.

  \subsection{Proof of Proposition \ref{P:X}.} The proof follows the same lines as the argument above, and so we shall be brief. The only substantial difference is that we cannot avail ourselves of the gain in $\epsilon$ given by \eqref{E:6.3.3}.
  
  It is enough to show that the analogue of \eqref{Eq:star} holds, but with $\Ctrei$ replaced by $\Ctrsei$, since the difference $\Ctrei-\Ctrsei$ can be handled as we did in verifying the hypotheses \eqref{E:schur-1} and \eqref{E:schur-2} for the error term $\Atrei$ in the proof of Proposition \ref{P:6.2.2}.
   First we have the analogue of \eqref{E:6.4.6} with $\Delta^s_\epsilon (w, z)$ replaced by the single terms 
  $\ge(w, z)\chi_s(w, z)$ and $ \bar{\ge}(z, w)\chi_s(w, z)$.
  
  Next, the argument of Lemma \ref{L:6.5.1} shows that 
  $$
  |\Ctrsei (f)(z)|\ \lesssim\ 1\ + \int\limits_{w\in\bndry D}\!\!\! |\nabla f(w)|\,\d (w, z)^{-2n + 2}\,d\lambda (w)
  $$
  whenever $f$ is in $C^1(\bndry D)$ and $|f(w)|\leq 1$ for any $w\in\bndry D$.
  
  Thus, by Lemma \ref{L:general},
  $$
   |\Ctrsei (f)(z)|\ \lesssim\ 1\quad\mbox{and}\quad \|\Ctrsei f\|_{L^2(\bndry D, d\lambda)}\lesssim 1
  $$
  for any normalized bump function $f$.
  The same conclusion holds for $\Ctrsei$ replaced by $(\Ctrsei)^*$. 
  
  At this point, an application of the $T(1)$ theorem, as before, completes the proof of the proposition.

  \section{The Cauchy-Szeg\H o projection with respect to a more general measure}\label{S:last}
  We now pass to the Cauchy-Szeg\H o projection $\Sw$ defined with respect to a measure of the form 
  $\omega\,d\sigma$
  described 
  at the beginning of Part II. Throughout this section we will write $\omega\,d\sigma =\varphi\,d\lambda$,
   where
   $\varphi$ (equivalently, $\omega =\varphi\Lambda$) is a continuous strictly positive density on $\bndry D$, and we will
  use the upper script $^\dagger$
   to designate the adjoint with respect to the inner product $(\cdot, \cdot)_\omega$ of
    $L^2(\bndry D, \omega\,d\sigma)=L^2(\bndry D, \varphi\,d\lambda)$. 
 
 Thus  $\Sw$ is the orthogonal
    projection of $L^2$ onto $\mathcal H^2$ in the sense that 
    $$\Swd = \Sw\, .$$
    
     Our goal is to prove Theorem \ref{T:6.1.2}, that is the analogue of Theorem \ref{T:6.1.1} in which $\S$ is replaced by $\Sw$.
    
    \subsection{Outline of the proof of Theorem \ref{T:6.1.2}}
    We begin by making two simple observations. First,
    \begin{equation}\label{E:7.2}
    \Sw(I+\Ctred -\Ctre)=\Ctre\, .
    \end{equation}
    This is the analogue of the identity {\em(c)} that followed from Proposition \ref{P:6.2.1}, and is proved in the same way as that result, when we
        use a corresponding version  of Proposition \ref{P:6.2.1}
         (also proved in \cite{LS-5}), in which the Leray-Levi measure $d\lambda$ is now replaced by its weighted version $\varphi\, d\lambda$,
   which lead to the identities 
    $$\Sw\Ctre =\Ctre\quad {\mathrm{and}}\quad \Ctre\Sw= \Sw\, .$$
      The second observation concerns any bounded operator $T$ on $L^2(\bndry D, \varphi\, d\l) \approx 
    L^2(\bndry D, d\l)$ and will be applied to $T=\Ctre$. Let $T^*$ denote the adjoint of $T$ with respect to the inner product 
    $$
    (f, g)_\lambda = \int\limits_{\bndry D} f\,\bar{g}\, d\l\, ,
    $$
    and $\Td$ its adjoint with respect to the inner product
    $$
    (f, g)_\omega = \int\limits_{\bndry D} f\,\bar{g}\ \varphi\,d\l\, .
    $$
      Then we have that 
    \begin{equation}\label{E:7.3}
    \Td =\varphi^{-1}\,T^*\,\varphi\, .
    \end{equation}
    In fact, by the definition of $T^\dagger$ and $T^*$ we have
     $$
      \int\limits_{\bndry D}\! f\,\bar{(T^\dagger g)}\, \varphi d\lambda =
  \int\limits_{\bndry D}\! (Tf)\,\bar{g}\, \varphi d\lambda = 
  \int\limits_{\bndry D}\! f\,\bar{T^* (g\, \varphi)} d\lambda.
    $$
  Since this holds for all $f\in L^2(\bndry D)$, we get $T^*(g\,\varphi)  = \Td(g)\varphi$, for all $g\in L^2(\bndry D)$, which is merely a restatement of \eqref{E:7.3}.
    
    The main thrust of the proof of Theorem \ref{T:6.1.2} is carried by the following proposition. For any (complex-valued) continuous function $\varphi$ on $\bndry D$, consider the commutator 
    $$
    [\Ctrei,\varphi] = \Ctrei\,\varphi - \varphi\,\Ctrei\, .
    $$
    \begin{Prop}\label{P:7.1}
    Suppose $\varphi$ and $\epsilon$ are given. There is 
        $\bar{\iota}(\epsilon)>0$ so that if
    $\iota\leq\bar{\iota}(\eps)$, then
    \begin{equation}\label{E:7.4}
    \|\,[\Ctrei, \varphi]\,\|_{_{L^p\to L^p}} \lesssim \eps M_p\, ,\quad \mbox{for}\ \ 1<p<\infty\, .
    \end{equation}
    \end{Prop}
    Here $M_p= p+ p/(p-1)$ is as in Proposition \ref{P:6.2.2}.
    We prove Proposition \ref{P:7.1} in Section \ref{SS:7.3} below, but now note that this assertion immediately implies the analogous statement with $\Ctrei$ replaced $(\Ctrei)^*$. Indeed, 
    $$
    [(\Ctrei)^*, \varphi] = (\Ctrei)^*\,\varphi -\varphi\,(\Ctrei)^* =
    \big(\bar\varphi\,\Ctrei -\Ctrei\,\bar\varphi\big)^* =
    -\big([\Ctrei,\bar\varphi]\big)^*.
    $$
   Thus \eqref{E:7.4} (with $\varphi$ replaced by $\bar\varphi$, and with $p$ and $M_p$ replaced by $p'$ and $M_{p'} = M_p$ respectively,  where $1/p+1/p'=1$) implies by duality that 
   \begin{equation}\label{E:7.5}
   \|\,[(\Ctrei)^*, \varphi]\,\|_{_{L^p\to L^p}}\lesssim \eps\, M_p,\quad \mbox{for}\ \  1<p<\infty\, .
   \end{equation}
   
   The significant consequence of \eqref{E:7.5} is the analogue of Proposition \ref{P:6.2.2}, in which
   $(\Ctrei)^*$ is replaced by $(\Ctrei)^\dagger$. Namely, for any $\eps>0$, there is an $\iota (\eps)>0$ so that $\iota\leq \iota (\eps)$ implies
   \begin{equation}\label{E:7.6}
   \|\,\Ctrei -\Ctreid\,\| _{_{L^p\to L^p}}\lesssim \eps^{1/2}\, M_p,\quad \mbox{for}\ \  1<p<\infty\, .
   \end{equation}
   
   To see why this is the case, recall by \eqref{E:7.3} that 
   $(\Ctrei)^\dagger = \varphi^{-1}(\Ctrei)^*\varphi$. Thus, 
   $
   \Ctrei -\Ctreid = \Ctrei - \varphi^{-1}(\Ctrei)^*\varphi\, .
   $
   But
   $$
   \varphi^{-1}(\Ctrei)^*\varphi - (\Ctrei)^*=
   \varphi^{-1}\big((\Ctrei)^*\varphi - \varphi(\Ctrei)^*\big)\,.
   $$
   So 
   $$
   \Ctrei -\Ctreid = 
   \Ctrei -(\Ctrei)^* +\varphi^{-1}[\varphi, (\Ctrei)^*]\, .
   $$
   Hence
   $$
   \|\,\Ctrei -\Ctreid\,\| _{_{L^p\to L^p}}\leq \|\,\Ctrei -(\Ctrei)^*\,\| _{_{L^p\to L^p}} +
   \sup|\varphi^{-1}|\,\|\,[\Ctrei,\,\varphi]\,\| _{_{L^p\to L^p}}\, .
   $$
   By Proposition \ref{P:6.2.2} and \eqref{E:7.5} the righthand side of the above is majorized by a multiple of $\eps^{1/2}M_p + \eps M_p$, which gives us \eqref{E:7.6}.
   
   With \eqref{E:7.6} in hand, the proof of Theorem \ref{T:6.1.2} follows the reasoning given in Section 
   \ref{SS:mainthm} after Proposition \ref{P:6.2.2}. Indeed, by \eqref{E:7.2} we get the analogue of 
   \eqref{E:6.2.1}, namely
   $$
   \Ctre +\Sw\Etreid - \Sw\Etrei = \Sw\big(I+\Ctrei-\Ctreid\big)\, ,
   $$
   since $\Ctre = \Ctrei + \Etrei$. The rest of the proof of of Theorem \ref{T:6.1.2} is then an almost word-for-word repetition of the argument after \eqref{E:6.2.1}. Here we invoke \eqref{E:7.6}, which guarantees that $I+\Ctrei -\Ctreid$ is invertible in $L^p$ when $\eps$ and $\iota$ are taken to be sufficiently small. With this, Theorem \ref{T:6.1.2} is proved (assuming the validity of Proposition \ref{P:7.1}).
   \subsection{A decomposition lemma}\label{SS:7.2}
   The proof of Proposition \ref{P:7.1} requires the following decomposition lemma, valid for a general class of operators. We consider a partition of $\Cn=\mathbb R^{2n}$ into disjoint cubes of side-length $\p$, given by
   $$
   \Cn = \cup_{k\in\mathbb Z^{2n}}Q^{\p}_k\, .
   $$
   When $\p =1$, we take $Q^1_k = Q^1_0 + k$, where $k$ ranges over the lattice points $\mathbb Z^{2n}$ of $\mathbb R^{2n}$, and $Q^1_0$ is the unit cube in $\mathbb R^{2n}$ with center at the origin. For general $\p>0$, we set $Q^{\p}_k = \p\,Q^1_k$. We say that the cube $Q^{\p}_k$
   {\em touches the cube} $Q^{\p}_j$ if $\bar{Q^{\p}_k}$ and $\bar{Q^{\p}_j}$ have non-empty intersection; equivalently, if $dist(Q^{\p}_j, Q^{\p}_k)\leq \p$ (with $dist$ the Euclidean distance).
   One notes that for each $k$ there are exactly $N=3^{2n}$ cubes $Q^{\p}_j$ that touch $Q^{\p}_k$. 
   These facts are easily verified by inspection when one scales back to the case $\p =1$. 
   
   We now revert to our domain $D\subset\Cn$. For a fixed $\p>0$ we write $\mathds{1}_{_k}$ for the characteristic function of $Q^{\p}_k\cap\bndry D$. Thus what has been said above implies that 
   $$
   \sum\limits_{j\in\mathbb Z^{2n}}\mathds{1}_{_j} =1\quad \mbox{on}\quad \bndry D\, ,\quad \mbox{and}
   $$
   \begin{equation}\label{E:7.7}
   \sum\limits_k\bigg(\sum\limits_{Q^{\p}_j\ \mbox{touches}\ Q^{\p}_k}\mathds{1}_{_j}\bigg) = N\quad\mbox{on}\quad \bndry D\, .
   \end{equation}
   \begin{Lem}\label{L:7.3} Suppose $T$ is a bounded operator on $L^p(\bndry D)$. 
   Fix $\p>0$
   
   \noindent and assume
   \medskip
   
   $(i)\quad \mathds{1}_{_k} T \mathds{1}_{_j}=0,$\qquad\qquad\quad if the the cubes $Q^{\p}_j$ and $Q^{\p}_k$ do not touch.
   \medskip
   
   $(ii)\quad\! \|\mathds{1}_{_k} T \mathds{1}_{_j}\|_{_{L^p\to L^p}}\leq A,$\quad otherwise.
   \medskip
   
 \noindent   Then 
   \begin{equation*}
   \|T\|_{_{L^p\to L^p}} \leq A\, N\, .
   \end{equation*}
   \end{Lem}
   Here $\mathds{1}_{_k} T \mathds{1}_{_j}(f) = \mathds{1}_{_k} T (f \mathds{1}_{_j})$.
   \begin{proof}We begin by observing that
   $$
   \int\limits_{\bndry D}|Tf|^p=\sum\limits_k\int\limits_{Q^{\p}_k\cap\bndry D}\!\! |Tf|^p\, .
   $$
   However, 
   $$
   \int\limits_{Q^{\p}_k\cap\,\bndry D}\!\! |Tf|^p =
   \int\limits_{\bndry D}\mathds{1}_{_k}\, |Tf|^p\, ,
   $$
   while 
   $$
   \mathds{1}_{_k} Tf = \mathds{1}_{_k}\sum\limits_{j} T(f\mathds{1}_{_j})\, ,
   $$
   since $\sum\mathds{1}_{_j}=1$ on $\bndry D$. Now by $(i)$, $\mathds{1}_{_k} T(f\mathds{1}_{_j}) =0$ unless $Q^{\p}_k$ touches 
   $Q^{\p}_j$, and so at most $N$ terms in the above sum are non-zero. Moreover we always have
   $\big|\sum a_j\big |^p\leq N^{p-1}|\sum |a_j|^p$, if at most $N$ of the $a_j$ are non-zero, as H\"older
    inequality shows. Thus
    $$
   \mathds{1}_{_k} |Tf|^p =
   \mathds{1}_{_k}\big|\sum\limits_jT(f\mathds{1}_{_j})\big|^p
   \leq N^{p-1}\mathds{1}_{_k}^p\!\!\!\sum\limits_{Q^{\p}_j\ \mbox{touches}\ Q^{\p}_k}
   \!\!\!|T(f\mathds{1}_{_j})|^p\, .
   $$
   Hence 
     $$
   \int\limits_{\bndry D}|Tf|^p\
   \leq\
   N^{p-1}\!\!\!\!\!\!\sum\limits_{Q^{\p}_j\ \mbox{touches}\ Q^{\p}_k}\ \ \ \int\limits_{\bndry D}
   |\mathds{1}_{_k}T(f\mathds{1}_{_j})|^p\ \leq	
   $$
   $$
   \leq
   N^{p-1}A^p\!\!\!\!\!
   \sum\limits_{Q^{\p}_j\ \mbox{touches}\ Q^{\p}_k}\ \ \int\limits_{\bndry D}
   |f|^p|\mathds{1}_{_j}|^p\, ,
   $$
   and by \eqref{E:7.7} and hypothesis {\em (ii)} the latter equals 
   $$
   N^{p-1} N A^p \sum\limits_j\int\limits_{\bndry D}|f|^p\mathds{1}_{_j}\ =\ 
   N^pA^p\int\limits_{\bndry D} |f|^p.
   $$
   The lemma is therefore proved.
   \end{proof}
   \subsection{Completion of the proof of Theorem \ref{T:6.1.2}}\label{SS:7.3}
   We will apply the lemma above to the operator $T= [\Ctrei, \varphi]$. So given a continuous function
   $\varphi$ and an $\eps>0$, we will show that we can pick a $\p>0$ and an $\iota (\epsilon)>0$, so that if $\iota\leq \iota (\eps)$, then $T$ satisfies the hypotheses {\em (i)} and {\em (ii)} of Lemma \ref{L:7.3}, with $A\approx \eps$.

   To begin with, we note that  if $z$ is not in the support of $f$, then
   \begin{equation}\label{E:7.star}
   \Ctrei(f)(z) = \int\limits_{w\in\bndry D}\!\!\!
   C_\epsilon (w, z)\chi_{_{s}}(w, z)f(w)\, d\l (w), 
   \end{equation}
   where $C_\epsilon (w, z)$ is the kernel of $\Ctre$ (see \eqref{E:Ctr-rep} and recall that $\Ctre$ is denoted $\Ctr$ in Section \ref{S:4}), and $\chi_{_{s}}(w, z)$ is the symmetrized cut-off function that occurred in \eqref{E:6.4.1a}.
  
  Next we consider  $\mathds{1}_{_k}T\mathds{1}_{_j}$ and observe that with the right choice of $\p$, the truncated operator 
   $\mathds{1}_{_k}T\mathds{1}_{_j}$ vanishes whenever $Q^{\p}_j$ and $Q^{\p}_k$ do not touch. Indeed when $Q^{\p}_j$ and $Q^{\p}_k$ do not touch, if $z\in Q^{\p}_k$ then $z$ does not lie in the support of 
   $f\mathds{1}_{_j}$. Hence by \eqref{E:7.star} we have
   $$
   \mathds{1}_{_k}(z)\,\Ctrei(f\mathds{1}_{_j})(z) =
   \mathds{1}_{_k}(z)\!\!\int\limits_{\bndry D} C_\epsilon (w, z)\chi_{_{s}}(w, z)f(w)\mathds{1}_{_j}(w)\, d\l (w)\, .
   $$
   There is a similar formula with $f$ replaced by $f\varphi$. Thus for either 
   $\mathds{1}_{_k}(z)\varphi (z)\Ctrei(f\mathds{1}_{_j})(z)$ or
    $\mathds{1}_{_k}(z)\Ctrei(f\mathds{1}_{_j}\varphi)(z)$ not to vanish, there must be a $z\in Q^{\p}_k$ and a $w\in Q^{\p}_j$ with $dist(w, z)\leq \p$, (given the support condition of $\chi_{_s}(w, z)$).
    
    But since $Q^{\p}_j$ and $Q^{\p}_k$ d not touch, necessarily $|z-w|>\p$, and this combined with the inequality $dist(w, z) = |w-z|\leq c\d(w, z)$, (see \eqref{E:2.14}) give a contradiction 
    to $\p\geq c\iota$. So we fix $\p$ with $\p = c\iota$, then this guarantees hypothesis {\em (i)}
     of the lemma, for our choice of $T$. Next we prove that if $\p$ and $\iota$ are chosen sufficiently small, then hypothesis {\em (ii)} is satisfied with $A\approx \eps$. To see this we let $z_k$ denote the center of the cube $Q^{\p}_k$. Then $z\in Q^{\p}_k$ implies that $|z-z_k|\leq \sqrt{k}\p$. Also if $Q^{\p}_k$ touches $Q^{\p}_j$, then $|w-z_k|\leq (\sqrt{k} + 1)\p$ whenever 
     $w\in Q^{\p}_j$. As a result, the uniform continuity of $\varphi$ grants that 
     \begin{equation}\label{E:7.9}
     \sup\limits_{z\in Q^{\p}_k}|\varphi (z) -\varphi (z_k)|\leq \eps \qquad \mbox{and}\qquad
     \sup\limits_{w\in Q^{\p}_j}|\varphi (w) -\varphi (z_k)|\leq \eps\, ,
     \end{equation}
    as long as $\p$ is made sufficiently small (in terms of $\eps$).
    
    Now write $\varphi = \varphi_k +\psi_k$, where $\varphi_k (z)= \varphi (z) -\varphi (z_k)$, and
    $\psi_k(z) =\varphi (z_k)$, for all $z$. Then 
    $[\Ctrei, \varphi] = [\Ctrei, \varphi_k] +   [\Ctrei, \psi_k]$. But $\psi_k$ is a constant, so
    $[\Ctrei, \psi_k]=0$, and hence
    $$
    \mathds{1}_{_k}T\mathds{1}_{_j} = \mathds{1}_{_k}[\Ctrei, \varphi]\mathds{1}_{_j} =
    \mathds{1}_{_k}[\Ctrei, \varphi_k]\mathds{1}_{_j} =
    \mathds{1}_{_k}\Ctrei (\varphi_k\mathds{1}_{_j}) - 
    \mathds{1}_{_k}\varphi_k\,\Ctrei\mathds{1}_{_j}\, .
    $$
   However   \ \ 
   $\sup\limits_{z\in\bndry D}|(\varphi_k\mathds{1}_{_j}) (w)|\leq 
   \sup\limits_{z\in Q^{\p}_k}|\varphi (z) -\varphi (z_k)|$. Therefore \eqref{E:7.9} gives 
    $$
    \|\mathds{1}_{_k}T\mathds{1}_{_j}f\|_{_{L^p}}\,\leq\, 2\eps\|\,\Ctrei\,\|_{_{L^p\to L^p}}
    \|f\|_{_{L^p}}.
    $$
    Thus hypothesis {\em (ii)} holds with $A= 2\eps\|\,\Ctrei\,\|_{_{L^p\to L^p}}$. Finally,
    Proposition \ref{P:X} ensures that $\|\Ctrei\|_{L^p\to L^p}\lesssim M_p$, if $\iota$ is sufficiently small. Thus Lemma \ref{L:7.3} implies $\|[\Ctrei, \varphi]\|_{L^p\to L^p}\lesssim \eps M_p$, thus proving Proposition \ref{P:7.1}. The proof of Theorem \ref{T:6.1.2} is thus now concluded.

\end{document}